\documentclass[12pt, reqno]{amsart}
\usepackage{amsmath, amsthm, amscd, amsfonts, amssymb, graphicx, color, mathrsfs}
\usepackage[bookmarksnumbered, colorlinks, plainpages]{hyperref}
\hypersetup{colorlinks=true,linkcolor=red, anchorcolor=green, citecolor=cyan, urlcolor=red, filecolor=magenta, pdftoolbar=true}

\newtheorem{theorem}{Theorem}[section]
\newtheorem{lemma}[theorem]{Lemma}
\newtheorem{proposition}[theorem]{Proposition}
\newtheorem{corollary}[theorem]{Corollary}
\theoremstyle{definition}
\newtheorem{definition}[theorem]{Definition}
\newtheorem{example}[theorem]{Example}

\theoremstyle{remark}
\newtheorem{remark}[theorem]{Remark}
\numberwithin{equation}{section}

\renewcommand{\AA}{\mathcal{A}}
\newcommand{\BB}{\mathcal{B}}

\newcommand{\EE}{\mathcal{E}}
\newcommand{\FF}{\mathcal{F}}

\newcommand{\HH}{\mathscr{H}}

\newcommand{\LL}{\mathcal{L}}

\newcommand{\PP}{\mathcal{P}}
\newcommand{\QQ}{\mathcal{Q}}

\newcommand{\UU}{\mathcal{U}}

\newcommand{\field}[1]{\mathbb{#1}}

\newcommand{\R}{\field{R}}
\newcommand{\N}{\field{N}}

\newcommand{\E}{\field{E}}
\renewcommand{\P}{\field{P}}

\newcommand{\pr}{{\text{\upshape{pr}}}}

\newcommand{\spn}{\mathop{\rm{span}}}

\newcommand{\ba}{{\rm{ba}}}

\newcommand{\al}{\alpha}

\newcommand{\de}{\delta}
\newcommand{\ep}{\varepsilon}

\renewcommand{\th}{\theta}

\newcommand{\la}{\lambda}

\newcommand{\si}{\sigma}

\newcommand{\om}{\omega}

\newcommand{\Om}{\Omega}

\newcommand{\sm}{\setminus}
\newcommand{\es}{\emptyset}

\renewcommand{\ba}{\mathop{\text{\upshape{ba}}}\nolimits}
\newcommand{\ca}{\mathop{\text{\upshape{ca}}}\nolimits}
\newcommand{\BUC}{\mathop{\text{\upshape{BUC}}}\nolimits}
\newcommand{\LSC}{\mathop{\text{\upshape{LSC}}}\nolimits}
\newcommand{\USC}{\mathop{\text{\upshape{USC}}}\nolimits}
\renewcommand{\hat}{\widehat}

\begin{document}

\setcounter{page}{1}

\title[Extension of nonlinear expectations]{Kolmogorov type and general extension results for nonlinear expectations}

\author[R.~Denk, M.~Kupper, \MakeLowercase{and} M.~Nendel]{Robert Denk,$^1$ Michael Kupper,$^2$ \MakeLowercase{and} Max Nendel$^3$}

\address{$^{1}$Department of Mathematics and Statistics, University of Konstanz, 78457 Konstanz Germany}
\email{\textcolor[rgb]{0.00,0.00,0.84}{Robert.Denk@uni-konstanz.de}}

\address{$^{2}$Department of Mathematics and Statistics, University of Konstanz, 78457 Konstanz Germany}
\email{\textcolor[rgb]{0.00,0.00,0.84}{kupper@uni-konstanz.de}}

\address{$^{3}$Department of Mathematics and Statistics, University of Konstanz, 78457 Konstanz Germany}
\email{\textcolor[rgb]{0.00,0.00,0.84}{Max.Nendel@uni-konstanz.de}}



\date{June 20, 2017}

\begin{abstract}
We provide extension procedures for nonlinear expectations to the space of all bounded measurable functions. We first discuss a maximal extension for convex expectations which have a representation in terms of finitely additive measures.
One of the main results of this paper is an extension procedure for convex expectations which are continuous from above and therefore admit a representation in terms of countably additive measures. This can be seen as a nonlinear version of the Daniell-Stone theorem. From this, we deduce a robust Kolmogorov extension theorem which is then used to extend nonlinear kernels to an infinite dimensional path space. We then apply this theorem to construct nonlinear Markov processes with a given family of nonlinear transition kernels.
\end{abstract}

\maketitle

\section{Introduction}
Given a set $M$ of bounded measurable functions $X\colon\Omega\to\mathbb{R}$ which contains the constants,
a nonlinear expectation is a functional $\EE\colon M\to\mathbb{R}$ which satisfies $\EE(X)\le \EE(Y)$ whenever $X(\omega)\le Y(\omega)$ for all $\omega\in \Omega$, and $\EE(\alpha 1_\Omega)=\alpha$ for all $\alpha\in\mathbb{R}$. If a nonlinear expectation $\EE$ is in addition sublinear, then $\rho(X):=\EE(-X)$, $X\in M$, is a coherent monetary risk measure as introduced by Artzner et al.~\cite{MR1850791} and Delbaen \cite{MR2011534},\cite{MR1929369}, see also F\"ollmer and Schied \cite{MR2779313} for an overview of convex monetary risk measures. Other prominent examples of nonlinear expectations include the g-expectation, see Coquet et al.~\cite{CHMP}, and the G-expectation introduced by Peng \cite{PengG},\cite{MR2474349}, see also Dolinsky et al. \cite{MR2868935} or Denis et al. \cite{MR2754968}. We also refer to Cheridito et al. \cite{MR2319056} and Soner et al. \cite{MR2746175},\cite{MR2842089} for the connection of the latter to fully non-linear PDEs and 2BSDEs.

The first part of this paper deals with the extension of a nonlinear expectation from a subspace $M$ to the space $\LL^\infty$ consisting of all bounded measurable functions $X\colon\Omega\to\mathbb{R}$. In line with  Maccheroni et al.~\cite{MR3153589}, we first show the existence of a maximal extension. In case that $\EE$ is convex on $M$, the maximal extension $\hat\EE$ is also convex, and has a dual representation in terms of finitely additive probability measures. We then focus on extensions which satisfy some additional continuity properties. If $\EE$ is convex and continuous from above on a Riesz subspace $M$, we construct an extension $\bar\EE$, which is continuous from below on $\LL^\infty$ and has a dual representation in terms of $\sigma$-additive probability measures (Theorem \ref{robustcara}). With the help
of Choquet's capacibility theorem \cite{MR0112844} we obtain the uniqueness of such an extension in a certain class of expectations. Thus, our extension result can be viewed as a generalization of the Daniell-Stone extension theorem, which states that a linear expectation $\EE$ which is continuous from above on a Riesz subspace $M$ has a unique linear extension $\bar \EE$ to $\LL^\infty$ over the $\sigma$-algebra $\sigma(M)$ generated by $M$. While for linear expectations the extension is still continuous from above, the same does not hold for convex expectations. Note that the continuity from above of a convex expectation $\EE$ on $\LL^\infty$ is a very strong condition which, in particular, implies that $\EE(X)=\EE(Y)$ whenever $X=Y$ $\mu$-almost surely for some probability measure $\mu$,
and that the representing probability measures in the dual representation of $\EE$ are dominated by $\mu$ as well.
However, nonlinear expectations are continuous from above on certain subspaces of $\LL^\infty$, see e.g.~Cheridito et al.~\cite{CKT} and the references therein. Hence, nonlinear expectations can be constructed by defining them on a subspace $M$ and extending them to $\LL^\infty$, the space of bounded $\sigma(M)$-measurable functions.

In the second part of the paper we illustrate this extension procedure in a Kolmogorov type setting. That is, for an arbitrary index set $I$ we construct nonlinear expectations on $\LL^\infty(S^I)$, where $S^I$ is the $I$-th product of a Polish space $S$. To that end, we first consider a family of expectations $\EE_J$ on linear subsets $M_J$ of  $\LL^\infty(S^J)$,
indexed by the set $\HH$ of all finite subsets of $I$. In line with Peng \cite{MR2143645}, under the natural consistency condition $\EE_K(f)=\EE_J(f\circ \pr_{JK})$ for every
$f\in M_K$ and all $J,K\in\HH$ with $K\subset J$, where $\pr_{JK}$ denotes the projection from $M_J$ to $M_K$, the family
$(\EE_J)$ can be extended to the space $M:=\{f\circ\pr_{J}\colon f\in M_J, \, J\in\HH \}$. Moreover,
if each $\EE_J$ is convex and continuous from above on $M_J$ the same also holds for the extension on $M$. Hence, by the general extension result, Theorem \ref{robustcara}, from the first part, there exists a convex expectation $\bar\EE$ on $\LL^\infty$ which is continuous from below, such that
$\bar\EE(f\circ \pr_J)=\EE_J(f)$ for all $f\in M_J$ and $J\subset I$ finite, see Theorem \ref{extension}. The corresponding dual version in the sublinear case leads to Theorem \ref{kolmogorov}, which is a robust version of Kolmogorov's extension theorem. We refer to Delbaen \cite{MR2276899}, Delbaen et al. \cite{MR2670421}, Cheridito et al. \cite{MR2199055}, F\"ollmer and Penner \cite{MR2323189} or Bartl \cite{dbartl} for a discussion on time consistency for dynamic monetary risk measures.

Finally, we construct consistent families $(\EE_J)$ of nonlinear expectations by means of nonlinear kernels, which are closely related to monetary risk kernels, as introduced by F\"ollmer and Kl\"uppelberg \cite{follmer2014spatial}. For two  subsets $M$ and $N$ of $\LL^\infty$, which contain the constants, a nonlinear kernel from $M$ to $N$ is a mapping $\EE\colon S\times M\to \R$ such that $\EE(x,\, \cdot\, )$ is a nonlinear expectation for all $x\in S$ and $\EE(\, \cdot\, , f)\in N$ for all $f\in M$. We then focus on nonlinear kernels, which map bounded continuous functions to bounded continuous functions in order to deal with stochastic optimal control problems, see e.g. Yong and Zhou \cite{MR1696772} or Fleming and Soner \cite{MR2179357}.

{\bf Notation.}
Let $(\Omega,\FF)$ be a measurable space and denote by $\LL^\infty(\Omega,\FF)$ the space of all bounded $\FF$-$\BB(\R)$-measurable random variables $X\colon \Omega\to \R$. Let $\ba (\Omega,\FF)$ be the space of all finitely additive signed measures of bounded variation on $(\Omega,\FF)$ containing the subset $\ca (\Omega,\FF)$ of all $\sigma$-additive signed measures. We denote by  $\ba_+(\Omega,\FF)$ the set of all positive elements in $\ba(\Omega,\FF)$, and by $\ba_+^1(\Omega,\FF)$ the set of those $\mu\in\ba_+(\Omega,\FF)$ with $\mu(\Omega)=1$. Analogously, we define $\ca_+(\Omega,\FF)$ and $\ca_+^1(\Omega,\FF)$.

Using the identification $\ba(\Omega,\FF)=(\LL^\infty(\Omega,\FF))'$ (cf. \cite{MR1009162}, p.~258), where $(\ldots)'$ stands for the topological dual space, we write $\mu X:= \int_\Om X\, {\rm d}\mu$ for
$\mu\in\ba (\Omega,\FF)$ and $X\in \LL^\infty(\Omega,\FF)$. The space $\LL^\infty(\Omega,\FF)$ and subspaces $M\subset \LL^\infty(\Omega,\FF)$ will always be endowed with the
supremum norm $\|\cdot\|_\infty$, and their dual spaces $\ba(\Omega,\FF)$ and $M'$ with the weak*-topology.
On subsets of these spaces we take the trace topology. On $\LL^\infty(\Omega,\FF)$ we consider the partial order
$X\ge Y$ whenever $X(\omega)\ge Y(\omega)$ for all $\omega\in \Omega$.
For $M\subset\LL^\infty(\Omega,\FF)$, we write $\alpha\in M$ and $\R\subset M$ if $\alpha  1_\Omega\in M$ and
$\{\alpha 1_\Omega: \alpha\in\R\}\subset M$, respectively, where $ 1_A$ is the indicator function of $A\in\FF$.

{\bf Structure of the paper.} In Section \ref{sec2} we study general extension results for nonlinear expectations and state their dual representations. The main extension results for convex expectations, which are continuous from above, are provided in Section \ref{sec3}. In Section \ref{sec4} we discuss nonlinear versions of Kolmogorov's extension theorem which are finally applied to nonlinear kernels in Section \ref{sec5}.


\section{General representation and extension results}\label{sec2}
In this section we introduce the basic definitions and state a maximal extension result for nonlinear expectations and their
dual representations. Throughout, let $M\subset\LL^\infty(\Om,\FF)$  with $\R\subset M$.
The following definition of a nonlinear expectation is due to Peng \cite{MR2143645}.
\begin{definition}\label{2.1}
A \emph{(nonlinear) pre-expectation} $\EE$ on $M$ is a functional $\EE\colon M\to \R$ which satisfies the following properties:
\begin{enumerate}
 \item[(i)] Monotonicity: $\EE(X)\leq \EE(Y)$ for all $X,Y\in M$ with $X\leq Y$.
 \item[(ii)] Constant preserving: $\EE(\al)=\al$ for all $\al\in \R$.
\end{enumerate}
A pre-expectation $\EE$ on $\LL^\infty(\Om,\FF)$ is called an \emph{expectation}.
\end{definition}
Note that a pre-expectation $\EE\colon M\to \R$ satisfies $|\EE(X)|\leq \|X\|_\infty$ for all $X\in M$.

The extension procedure of positive linear functionals by Kantorovich (cf.~\cite{MR0224522}, p.~277) indicates
the following extension of a pre-expectation $\EE\colon M\to\R$ to an expectation $\hat\EE\colon \LL^\infty(\Omega,\FF)\to\R$.
For related extension results on niveloids we refer to Maccheroni et al.~\cite{MR3153589}.

\begin{proposition}\label{2.3}\label{bacara}
 For a pre-expectation $\EE\colon M\to \R$, define \[\hat\EE(X):=\inf\{\EE(X_0)\colon X_0\in M, X_0\geq X\}\]
for all $X\in \LL^\infty(\Om,\FF)$.
\begin{enumerate}
 \item[a)] $\hat\EE\colon \LL^\infty(\Om,\FF)\to \R$ is the maximal expectation with $\hat\EE|_M=\EE$, i.e. $\hat \EE(X)=\EE(X)$ for all $X\in M$ and for every expectation $\tilde \EE\colon \LL^\infty(\Om,\FF)\to \R$ with $\tilde \EE|_M=\EE$ we have that $\tilde \EE(X)\leq \hat \EE(X)$ for all $X\in \LL^\infty(\Om,\FF)$.
 \item[b)] If $M$ is convex and $\EE$ is convex, then $\hat\EE$ is convex.
 \item[c)] If $M$ is a convex cone and $\EE$ is sublinear, then $\hat\EE$ is sublinear.
\end{enumerate}
\end{proposition}

\begin{proof}
 a) Let $X\in \LL^\infty(\Om,\FF)$. Note that $\hat\EE(X)>-\infty$  since for each $X_0\in M$ with $X_0\ge X$ one has $\EE(X_0)\ge \EE(-\|X\|_\infty) = -\|X\|_\infty$. On the other hand, $\|X\|_\infty\in M$ implies $\hat\EE(X)\le \EE(\|X\|_\infty)=\|X\|_\infty$. So $\hat\EE(X)$ is finite.

  If $X\in M$, we have $\EE(X)\leq \EE(X_0)$ for all $X_0\in M$ with $X_0\geq X$, i.e. $\hat\EE(X)=\EE(X)$. Since $\R\subset M$, we thus
 obtain $\hat\EE(\al)=\al$ for all $\al\in \R$. Now let $X,Y\in \LL^\infty(\Om,\FF)$ with $X\leq Y$. Then $Y_0\geq X$ for all $Y_0\in M$ with $Y_0\geq Y$,
 and therefore $\hat\EE(X)\leq \hat\EE(Y)$.

If $\tilde\EE\colon \LL^\infty(\Om,\FF)\to \R$ is another
 expectation with $\tilde\EE|_M=\EE$ and $X\in \LL^\infty(\Om,\FF)$, then
 \[
  \tilde\EE(X)\leq \tilde\EE(X_0)=\EE(X_0)
 \]
 for all $X_0\in M$ with $X_0\geq X$. Hence, $\tilde\EE(X)\leq \hat\EE(X)$.

The statements b) and c) follow directly from the definition of  $\hat\EE$.
\end{proof}

\begin{remark}\label{minexp}
For a pre-expectation $\EE\colon M\to \R$, let $\check\EE(X):=\sup\{\EE(X_0)\colon X_0\in M, X_0\leq X\}$ for all $X\in \LL^\infty(\Om,\FF)$. Then, one readily verifies that $\check\EE\colon \LL^\infty(\Om,\FF)\to \R$ is the smallest expectation which extends $\EE$. However, convexity of $\EE$ usually does not carry over to $\check\EE$.
\end{remark}

Throughout the remainder of this section, let $M\subset\LL^\infty(\Omega,\FF)$ be a linear subspace with $1\in M$. In this case, we can give an explicit description of $\hat\EE$, using tools from convex analysis and duality theory. For a convex function $\EE\colon M\to \R$, let $\EE^*$ be its conjugate function
\[
\EE^*(\mu) := \sup_{X\in M} \big( \mu X-\EE(X)\big)
\]
where $\mu\colon M\to\R$ is a linear functional. We start with the well-known representation of convex pre-expectations on $M$. For the sake of completeness we give a proof in the Appendix \ref{appA}.

\begin{lemma}\label{2.5}\label{compactness}
 Let  $\EE\colon M\to \R$ be a convex pre-expectation. Then, every linear functional $\mu\colon M\to \R$ with $\EE^*(\mu)<\infty$ is a linear pre-expectation.
 Further, one has
 \begin{equation}\label{rep:M}
\EE(X) = \max_{\mu\in M'}\big( \mu X-\EE^*(\mu)\big)\quad \mbox{for all }X\in M,
\end{equation}
where
the maximum is attained on the convex compact set $\{\mu\in M': \EE^*(\mu)\le \alpha\}$ for every $\alpha\ge  \|X\|_\infty-\EE(X)$.
If $\EE$ is sublinear, then $\EE^*(\mu)<\infty$ implies that $\EE^*(\mu)=0$ for all $\mu\in M'$ and we obtain that
 \[ \EE(X) = \max \big\{  \mu X: \mu\in M',\, \EE^*(\mu)=0\big\}.\]
\end{lemma}

The previous lemma shows that a convex pre-expectation has the \emph{translation property}, i.e.
$\EE(X+\al)=\EE(X)+\al$ for all $X\in M$ and $\al\in\R$. In particular, $\EE$ is $1$-Lipschitz continuous and $\rho(X):=\EE(-X)$ defines a convex risk measure on $M$.
For a discussion of risk measures we refer to F\"ollmer and Schied \cite{MR2779313} and the references therein.

\begin{remark}\label{2.6}
  We apply Lemma~\ref{2.5} to the linear case. Let $\mu\in M'$ be a linear pre-expectation. Then,
  \begin{equation}\label{2-2}
   \hat \mu(X) = \max\{ \nu X: \nu\in\ba_+^1(\Omega,\FF),\, \nu|_M=\mu\}\quad \mbox{for all }X\in\LL^\infty(\Omega,\FF).
   \end{equation}
  In fact, by Lemma~\ref{2.3} c) $\hat \mu$ is a sublinear expectation, and an application of Lemma~\ref{2.5} with $M=\LL^\infty(\Omega,\FF)$ yields
  \[ \hat\mu(X) = \max\{\nu X: \nu\in \ba (\Omega,\FF),\, \hat\mu^*(\nu)=0\}.\]
  For each $\nu\in\ba(\Omega,\FF)$ with $\hat\mu^*(\nu)=0$, another application of Lemma~\ref{2.5}  implies $\nu\in\ba_+^1(\Omega,\FF)$, and from
  $\hat\mu(X)\ge \nu X$ for all $X\in\LL^\infty(\Omega,\FF)$ we see that $\nu X_0\le\hat\mu(X_0)=\mu(X_0)$ for all $X_0\in M$. As $M$ is a linear subspace it follows that $\nu|_M=\mu$. This implies ``$\le$'' in \eqref{2-2}. On the other hand, each $\nu\in\ba_+^1(\Omega,\FF)$ with $\nu|_M=\mu$ is an expectation extending $\mu$ and therefore $\hat\mu(X)\ge \nu X$ by the maximality of $\hat\mu$.
  \unskip\nobreak\hfill$\Diamond$\par\addvspace{\medskipamount}
\end{remark}

\begin{theorem}\label{2.7}
  Let $\EE\colon M\to \R$ be a convex pre-expectation on $M$. Then, the maximal extension $\hat \EE$ has the representation
  \begin{equation}\label{repEhat}
  \hat \EE(X) = \max_{\substack{\mu\in M'\\ \EE^*(\mu)<\infty}} \big( \hat\mu(X) - \EE^*(\mu)\big) = \max_{\nu\in\ba_+^1(\Omega,\FF)} \big( \nu X - \EE^*(\nu|_M)\big)\
  \end{equation}
  for all $X\in \LL^\infty(\Omega,\FF)$. Moreover, $\hat\EE^\ast(\nu)=\EE^\ast(\nu|_M)$ for all $\nu\in \ba_+^1(\Omega,\FF)$.
\end{theorem}

\begin{proof}
By Lemma \ref{2.5}, we have that every $\mu\in M'$ with $\EE^*(\mu)<\infty$ is a linear pre-expectation on $M$ and therefore, $\hat\mu$ is well-defined. Let $X\in \LL^\infty(\Om,\FF)$. By the maximality of $\hat\EE$ we have that
\[
 \sup_{\substack{\mu\in M'\\ \EE^*(\mu)<\infty}} \big( \hat\mu(X)-\EE^*(\mu)\big)\leq \hat\EE(X)
\qquad
\mbox{and}
\qquad
  \sup_{\nu\in\ba_+^1(\Omega,\FF)} \big( \nu X - \EE^*(\nu|_M)\big)\leq \hat\EE(X)
\]
as the left-hand sides are expectations extending $\EE$. By Lemma \ref{2.5} applied to $\hat\EE$ and $M=\LL^\infty(\Om,\FF)$, there exists a linear expectation $\nu\in \ba_+^1(\Om,\FF)$ with $\hat\EE^*(\nu)<\infty$ and $\hat\EE(X)=\nu X-\hat\EE^*(\nu)$. Then, $\mu:=\nu|_M\in M'$ is a linear pre-expectation with $\EE^*(\mu)\leq \hat\EE^*(\nu)<\infty$. By Remark \ref{2.6} we get that
\[
 \hat\EE(X)=\nu X-\hat\EE^*(\nu)\leq \hat\mu (X)-\EE^*(\mu)\leq \hat\EE(X).
\]
It remains to show that $\hat\EE^*(\nu)=\EE^*(\nu|_M)$ for all $\nu\in \ba_+^1(\Om,\FF)$. Clearly, $\EE^*(\nu|_M)\leq \hat\EE^*(\nu)$. In order to show the inverse inequality, let $\nu\in \ba_+^1(\Om,\FF)$ with $\EE^*(\nu|_M)<\infty$, $X\in \LL^\infty(\Om,\FF)$ and $\ep>0$. Then, there exists some $X_0\in M$ with $X_0\geq X$ and $\EE(X_0)\leq \hat\EE(X)+\ep$. Hence, we get that
\[
 \nu X-\hat\EE(X)\leq \nu X-\hat\EE(X_0)+\ep\leq\nu X_0-\hat\EE(X_0)+\ep \leq \EE^*(\nu|_M)+\ep.
\]
Letting $\ep\searrow 0$, we obtain that $\nu X-\hat\EE(X)\leq \EE^*(\nu|_M)$ and the proof is complete.
\end{proof}


\section{Continuous extensions of nonlinear expectations}\label{sec3}
Although the maximal extension $\hat\EE$ is rather straightforward, its representation \eqref{repEhat}
is in terms of finitely additive measures in $\ba_+^1(\Omega,\FF)$. In this section we focus on an alternative extension
admitting a representation with probability measures in $\ca_+^1(\Omega,\FF)$. Throughout this section, let $M\subset\LL^\infty(\Om,\FF)$ be a linear subspace with $1\in M$.

\begin{definition}
Let $\EE\colon M\to \R$ be a pre-expectation on $M$.
\begin{enumerate}
 \item[a)] We say that $\EE$ is \textit{continuous from above} if $\EE(X_n)\searrow \EE(X)$
for all $(X_n)_{n\in \N}\in M^{\mathbb N}$ and $X\in M$ with $X_n\searrow X$ as $n\to \infty$.
 \item[b)] We say that $\EE$ is \textit{continuous from below} if $\EE(X_n)\nearrow \EE(X)$
for all $(X_n)_{n\in \N}\in M^{\mathbb N}$ and $X\in M$ with $X_n\nearrow X$ as $n\to \infty$.
\end{enumerate}
\end{definition}

In Fan \cite{MR0055678}, a function $f\colon E\times F\to \R$ defined on arbitrary sets $E$ and $F$ is said to be convex on $F$ if for all $y_1,y_2\in F$ and $\lambda\in [0,1]$ there exists an element $y_0\in F$ such that
\[
f(x, y_0)\le \lambda f(x,y_1)+(1-\lambda) f(x,y_2)\quad \mbox{for all }x\in E.
\]
Analogously, concavity on $E$ is defined. By Fan's minimax theorem (\cite{MR0055678}, Theorem 2) one has
\[
\max_{x\in E} \inf_{y\in F} f(x,y) = \inf_{y\in F} \max_{x\in E} f(x,y)
\]
if $E$ is a compact Hausdorff space, $f(\cdot,y)$ is upper semicontinuous on $E$ for each $y\in F$, $f$ is convex on $F$ and concave on $E$.

\begin{lemma}\label{contabove}
Let $\EE\colon M\to \R$ be a convex pre-expectation. Then, $\EE$ is continuous from above if and only if every $\mu\in M^\prime$ with $\EE^*(\mu)<\infty$ is continuous from above.
\end{lemma}
\begin{proof} Recall that by Lemma \ref{2.5} one has
$\EE(Y) = \max_{\mu\in M'}\big( \mu Y-\EE^*(\mu)\big)$ for all $Y\in M$, and
let $(X_n)_{n\in \N}\in M^\N$ such that $X_n\searrow X$ for some $X\in M$.

If $\EE$ is continuous from above, then for every $\mu\in M^\prime$ with $\EE^*(\mu)<\infty$ and all $\la>0$ one has
\begin{align*}
 0&\leq \mu X_n -\mu X=\mu(X_n-X)=\la^{-1}\mu(\la (X_n-X))\\
 &\leq \la^{-1} \EE(\la(X_n-X))+\la^{-1} \EE^*(\mu) \searrow \la^{-1} \EE^*(\mu),\quad \mbox{as }n\to \infty.
\end{align*}
Letting $\la \to\infty$, we thus obtain that $\inf_{n\in \N}\mu X_n =\mu X$.

Conversely, if every $\mu\in M^\prime$ with $\EE^*(\mu)<\infty$ is continuous from above,
we apply Fan's minimax theorem with $E:=\{\mu\in M':\EE^*(\mu)\le \|X_1\|_\infty-\EE(X)\}$,
$F:= \mathbb{N}$ and $f\colon E\times F\to \R,\; (\mu,n)\mapsto \mu X_n-\EE^*(\mu)$, and obtain
\begin{align*}
  \EE(X)&=\max_{\mu\in E}\left(\mu X-\EE^*(\mu)\right)=\max_{\mu\in E}\inf_{n\in \N}\left(\mu X_n-\EE^*(\mu)\right)\\
  &=\inf_{n\in \N}\max_{\mu\in E}\left(\mu X_n-\EE^*(\mu)\right)=\inf_{n\in \N}\EE(X_n),
 \end{align*}
so that $\EE$ is continuous from above.
\end{proof}

\begin{remark}
Let $\EE\colon \LL^\infty(\Om,\FF)\to \R$ be a convex expectation which is continuous from above. Then, Lemma \ref{compactness} and Lemma \ref{contabove} imply that
\[
  \PP_n:=\{\mu\in \ca_+^1(\Om,\FF)\colon \EE^*(\mu)\leq n\}
\]
is a compact subset of $\ba(\Om,\FF)$ for all $n\in \N$. Hence, for all $n\in \N$ there exists a probability measure $\nu_n\in \ca_+^1(\Om,\FF)$ such that all $\mu\in \PP_n$ are
$\nu_n$-continuous and the family $\big\{\frac{{\rm d}\mu}{{\rm d}\nu_n}\colon \mu\in \PP_n\big\}$ is uniformly integrable (cf. \cite{MR2267655}, p. 291). Therefore, every
$\mu\in \ca_+^1(\Om,\FF)$ with $\EE^*(\mu)<\infty$ is absolutely continuous w.r.t. the probability measure $\nu:=\sum_{n=1}^\infty2^{-n}\nu_n$. By Lemma \ref{compactness} and
Lemma \ref{contabove} we have that
\[
 \EE(X)=\EE(Y)
\]
for all $X,Y\in \LL^\infty(\Om,\FF)$ with $X=Y$ $\nu$-almost surely.
\unskip\nobreak\hfill$\Diamond$\par\addvspace{\medskipamount}
\end{remark}

Since the continuity from above on $\LL^\infty(\Om,\FF)$ of a convex expectation $\EE$ already implies that $\EE$ is dominated by some reference measure, this assumption is too strong in many applications. This motivates the following

\begin{definition}
 Let $\EE\colon \LL^\infty(\Om,\FF)\to \R$ be a convex expectation. Then, we say that $(\Om,\FF,\EE)$ is a \textit{convex expectation space} if there exists a set of probability measures $\PP\subset \ca_+^1(\Om,\FF)$ such that
 \[
  \EE (X)=\sup_{\mu\in \PP}\big(\mu X-\EE^*(\mu)\big)
 \]
 for all $X\in \LL^\infty(\Om,\FF)$. If in addition $\EE$ is sublinear, then $(\Om,\FF,\EE)$ is called a \textit{sublinear expectation space}.
\end{definition}

The following proposition is a standard result which shows that in a topological space $\Om$ tightness is sufficient to at least obtain continuity from above on $C_b(\Om)$. For the reader's convenience we provide a proof of this statement.

\begin{proposition}\label{tightness}
Let $\Om$ be a topological space with Borel $\si$-algebra $\FF$ on $\Om$, and let $\EE\colon C_b(\Om)\to \R$ be given by
\[
 \EE(X):=\sup_{\mu\in \PP}\mu X \quad \mbox{for all }X\in C_b(\Om),
\]
where $\PP\subset \ca_+^1(\Om,\FF)$ is tight.
Then, the sublinear pre-expectation $\EE$ is continuous from above.
\end{proposition}

\begin{proof}
Let $(X_n)_{n\in \N}$ be a sequence in $C_b(\Om)$ with $X_n\searrow 0$ and $\ep>0$. We may w.l.o.g.~assume that $X_1\neq 0$. As $\PP$ is tight, there exists a compact set $K\subset \Om$ such that
\[
 \sup_{\mu\in \PP} \mu(\Om \setminus K)\leq \ep\|X_1\|_\infty^{-1}.
\]
 By Dini's lemma, we have that $\|X_n1_K\|_\infty\to 0$. Hence,
\begin{align*}
 \EE(X_n)&\leq \sup_{\mu\in \PP}\mu(X_n1_K)+  \sup_{\mu\in \PP}\mu(X_n(1-1_K))\\
 &\leq \|X_n1_K\|_\infty+\|X_1\|_\infty \sup_{\mu\in \PP}\mu (\Om \setminus K)\\
 &\leq \|X_n1_K\|_\infty+\ep\to \ep, \quad \mbox{as }n\to \infty.
\end{align*}
Letting $\ep\searrow 0$, we obtain that $\lim_{n\to \infty}\EE(X_n)=0$ and therefore, $\EE$ is continuous from above at $0$. Since $\EE$ is sublinear, it is continuous from above.
\end{proof}

The following Lemma is a variant of Extensions du th\'eor\`eme 1 a) in Choquet \cite{MR0112844}.

\begin{lemma}\label{extbelow}
 Let $\FF\subset 2^\Om$ be a $\si$-algebra and $\EE\colon \LL^\infty(\Om,\FF)\to \R$ be continuous from below. Then, $\hat\EE\colon \LL^\infty(\Om,2^\Om)\to \R$ is continuous from below as well.
\end{lemma}

\begin{proof}
Let $X\in \LL^\infty(\Om,2^\Om)$ and $(X_n)_{n\in \N}$ be a sequence in $\LL^\infty(\Om,2^\Om)$ with $X_n\nearrow X$. Fix $\ep>0$. Then, for every $n\in \N$, there exists an $X_0^n\in \LL^\infty(\Om,\FF)$ with $X_n\leq X_0^n\leq \|X\|_\infty$ and
 \[
  \EE(X_0^n)\leq \hat\EE(X_n)+\ep
 \]
for all $n\in \N$. Define $Y_n:=\inf_{k\geq n}X_0^k$. Then, $Y_n\in \LL^\infty(\Om,\FF)$ with $X_n\leq Y_n\leq Y_{n+1}\leq \|X\|_\infty$ and
\[
 \EE(Y_n)\leq \EE(X_0^n)\leq \hat\EE(X_n) + \ep
\]
for all $n\in \N$. As $X_n\leq Y_n\leq \|X\|_\infty$ for all $n\in \N$, we get that $Y:=\sup_{n\in \N}Y_n\in \LL^\infty(\Om,\FF)$ with
$
 X=\sup_{n\in \N}X_n\leq\sup_{n\in \N}Y_n=Y
$
 and $Y_n\nearrow Y$. Since $\EE$ is continuous from below, we obtain that
\[
 \hat\EE(X)\leq \EE(Y)=\lim_{n\to \infty}\EE(Y_n)\leq \lim_{n\to \infty}\hat\EE(X_n)+\ep.
\]
Letting $\ep\searrow 0$, we obtain that $\hat\EE(X)\leq \lim_{n\to \infty}\hat\EE(X_n)$ and therefore, $\hat\EE(X)=\lim_{n\to \infty}\hat\EE(X_n)$.
\end{proof}

Let $\FF\subset 2^\Om$ be a $\si$-algebra. For a convex expectation $\EE\colon \LL^\infty(\Om,\FF)\to \R$ which is continuous from below, the following example shows that in general, there exists not even one $\mu\in \ca_+^1(\Om,\FF)$ with $\EE^*(\mu)<\infty$. However, if $\EE$ is dominated by some reference measure $\nu\in \ca_+^1(\Om,\FF)$, i.e. $\EE(X)=\EE(Y)$ for all $X,Y\in\LL^\infty(\Om,\FF)$ with $X=Y$ $\nu$-almost surely, then $\EE$ can even be represented by probability measures which are absolutely continuous w.r.t. $\nu$ (cf. \cite{MR2779313}, Theorem 4.33).

\begin{example}\label{counterexample}
Let $\Om$ be a set of cardinality $|\Om|=\aleph_1$. Let
\[
 \AA:=\{A\in 2^\Om\colon |A|=\aleph_0\; {\rm or}\; |\Om\sm A|=\aleph_0\}
\]
and
\[
 \la(A):=\begin{cases}
          0, & |A|=\aleph_0,\\
          1, & |\Om\sm A|=\aleph_0
         \end{cases}
\]
for all $A\in \AA$. Then $(\Om,\AA,\la)$ is a probability space and by Proposition \ref{bacara} and Lemma \ref{extbelow}, $\hat\la\colon \LL^\infty(\Om,2^\Om)\to \R$ is a sublinear expectation which is continuous from below and extends $\la$. However, by a result due to Bierlein \cite[Satz 1C]{MR0149533}, there exists no $\mu\in \ca_+^1(\Om,2^\Om)$ with $\mu|_\AA=\la$. Hence, by Theorem \ref{2.7}, there exists no $\mu\in \ca_+^1(\Om,2^\Om)$ with $\hat\la^*(\mu)<\infty$. Assuming the continuum hypothesis, we may choose $\Om=[0,1]$ and $\la\colon \AA\to \R$ as the restriction of Lebesgue measure to $\AA$.
\end{example}

For $X,Y\in \LL^\infty(\Om,\FF)$ let $(X\wedge Y)(\om):=\min\{X(\om),Y(\om)\}$ and $(X\vee Y)(\om):=\max \{X(\om),Y(\om)\}$ for all $\om \in \Om$. For the remainder of this section we assume that the linear subspace
$M$ of $\LL^\infty(\Om,\FF)$ is a \textit{Riesz subspace}, i.e.~$M\wedge M=M$ or equivalently $M\vee M=M$. Here, $M\wedge M$ and $M\vee M$ are the sets of all $X\wedge Y$ and
$X\vee Y$ with $X,Y\in M$, respectively. Typical examples for Riesz subspaces of $\LL^\infty(\Om,\FF)$ are:
\begin{enumerate}
 \item[(i)] The space $\spn\{1_A\colon A\in \AA\}$ of all $\AA$-step functions, where $\AA\subset \FF$ is an algebra.
 \item[(ii)] The space $C_b(\Om)$ of all continuous bounded functions $\Om\to \R$, if $\Om$ is a toplogical space and $\FF$ is the Borel $\si$-algebra on $\Om$.
\end{enumerate}
Denote by $M_\si$ and $M_\de$ the set of all $X\in \LL^\infty(\Om,2^\Om)$ for which there exists a sequence $(X_n)_{n\in \N}\in M^{\mathbb N}$ with $X_n\nearrow X$ and $X_n\searrow X$, respectively.
In the sequel, we will use the following version of Choquet's capacibility theorem (cf.~\cite{MR0112844}, Th\'eor\`eme 1).
Let $\EE\colon \LL^\infty(\Om,2^\Om)\to \R$ be an expectation and $M$ a  Riesz subspace with $1\in M$.
If $\EE$ is continuous from below and $\EE|_{M_\de}$ is continuous from above, then for all $X\in \LL^\infty(\Om,\si(M))$ one has
\[
 \EE(X)=\sup\{\EE(X_0)\colon X_0\in M_\de, X_0\leq X\}.
\]
This follows from \cite[Extensions du th\'eor\`eme 1, 2)]{MR0112844} and the monotone class theorem (\cite{MR521810}, Chapter I, (22.3)).

By the Daniell-Stone theorem, for every linear pre-expectation $\mu\colon M\to\R$ which is continuous from above, there exists a unique expectation $\nu\in \ca_+^1(\Om,\si(M))$ which is continuous from above and extends $\mu$, i.e. $\mu X=\int X d\nu$ for all $X\in M$. However, in the sublinear case, a similar statement does not hold, as illustrated by the following example. For a convex version of the Daniell-Stone theorem and the respective representation results we refer to Cheridito et al.~\cite{CKT} and the references therein.

\begin{example}
Let $\Om:=[0,1]$ and $\EE(X):=\max_{\om\in \Om} X(\om)$ for all $X\in M:=C(\Om)$. By Dini's lemma $\EE\colon M\to \R$ is continuous from above, and thus has the representation
\[
\EE(X)=\max_{\mu\in \ca_+^1([0,1],\FF)} \mu X \quad \mbox{for all }X\in M,
\]
where $\FF$ denotes the Borel $\sigma$-algebra on $[0,1]$. Notice that $\ca_+^1(\Om,\FF)$ is compact in $\ca(\Om,\FF)=C(\Om)'$ or equivalently tight by Prokhorov's theorem, however it is not compact in $\ba(\Om,\FF)=\LL^\infty(\Om,\FF)'$.
Suppose there existed an expectation $\tilde\EE:\LL^\infty(\Om,\FF)\to\R$ which extends $\EE$ and is continuous from above. Approximating the upper semicontinuous indicator function $1_{\{\om\}}$ with continuous functions from above implies $\tilde\EE( 1_{\{\om\}})\geq 1$ for all $\om\in \Om$. Hence, for every sequence $(A_n)_{n\in \N}\subset \FF$ with $A_n\neq\emptyset$ and $ 1_{A_n}\searrow 0$, one has $\tilde\EE( 1_{A_n})\geq 1$ for all $n\in\N$ and $\tilde\EE(0)=0$.
\end{example}

The main theorem of this section, Theorem \ref{robustcara}, states that for every convex pre-expectation $\EE\colon M\to \R$ which is continuous from above, there exists exactly one expectation $\bar\EE\colon \LL^\infty(\Om,\si(M))\to \R$ which is continuous from below on $\LL^\infty(\Om,\si(M))$ and continuous from above on $M_\de$. Moreover, $\bar\EE$ is convex and admits a representation in terms of probablity measures on $(\Om,\si(M))$.

We start by extending a pre-expectation $\EE\colon M\to \R$ which is continuous from above to a pre-expectation $\EE\colon M_\de\to \R$ which is continuous from above. For related results in the context of robust pricing and hedging in financial markets we refer to Cheridito et al.~\cite{CKT2}.

\begin{lemma}\label{edelta}
 Let $\EE\colon M\to \R$ be a pre-expectation which is continuous from above. Then, there exists a unique pre-expectation $\EE_\de\colon M_\de\to \R$ which is continuous from above and extends $\EE$. Moreover, 
 $\EE_\de=\hat\EE|_{M_\de}$, i.e.~$\EE_\de$ is the largest pre-expectation $\tilde\EE\colon M_\de\to \R$ with $\tilde\EE|_M=\EE$.
\end{lemma}

\begin{proof}
Let $X\in M_\de$ and $(X_n)_{n\in \N}\in M^{\mathbb N}$ with $X_n\searrow X$. Then, $\EE(X_n)\geq -\|X\|_\infty$
for all $n\in \N$. Define
\[
 \EE_\de(X):=\lim_{n\to \infty}\EE(X_n)>-\infty.
\]
First, we show that $\EE_\de(X)$ is independent of the sequence $(X_n)_{n\in \N}\in M^{\mathbb N}$. Let $(Y_n)_{n\in \N}\in M^{\mathbb N}$ with $Y_n\searrow X$. Then, $Z_n^k:=X_n\vee Y_k\in M$ for all $k,n\in \N$ and $Z_n^k\searrow Y_k$ as $n\to \infty$ for all $k\in \N$. Hence, as $\EE$ is continuous from above, we get that
\[
 \EE(Y_k)=\lim_{n\to \infty}\EE(Z_n^k)\geq \lim_{n\to \infty}\EE(X_n)
\]
for all $k\in \N$. Thus, $\lim_{n\to \infty}\EE(Y_n)\geq \lim_{n\to \infty}\EE(X_n)$ and therefore
$\lim_{n\to \infty}\EE(Y_n)=\lim_{n\to \infty}\EE(X_n)$ by symmetry, which shows that $\EE_\de$ is well-defined. Clearly, $\EE_\de$ defines a pre-expectation on $M_\de$ with $\EE_\de|_M=\EE$.

Now, let $X\in M_\de$ and $(X_n)_{n\in \N}\in M^{\mathbb N}_\de$ with $X_n\searrow X$. For all $n\in \N$ let $(X_n^k)_{k\in \N}\in M^{\mathbb N}$ with $X_n^k\searrow X_n$ as $k\to \infty$. Define $Y_n:=X_1^n\wedge\ldots\wedge X_n^n$ for all $n\in \N$. Then, as $M$ is directed downwards, we have that $Y_n\in M$ with $Y_n\geq Y_{n+1}$ for all $n\in \N$. Moreover, $Y_n\geq X_n$ for all $n\in \N$ and $X_n^k\geq Y_n$ for all $k,n\in \N$ with $k\leq n$. Hence,
\[
X_m= \lim_{k\to \infty}X_m^k\geq \lim_{n\to \infty}Y_n\geq \lim_{n\to \infty}X_n=X
\]
for all $m\in \N$. Altogether, $Y_n\searrow X$ with $Y_n\geq X_n$ for all $n\in \N$ and therefore
\[
 \EE_\de(X)=\lim_{n\to \infty}\EE(Y_n)\geq \lim_{n\to \infty}\EE_\de(X_n).
\]
As $\EE_\de(X_n)\geq \EE_\de(X)$ for all $n\in \N$, we obtain that $\EE_\de(X)=\lim_{n\to \infty}\EE_\de(X_n)$.

We have that $\hat\EE(X)\geq \EE_\de(X)$ for all $X\in M_\de$ as $\hat\EE$ is the largest expectation which extends $\EE$. Let $(X_n)_{n\in \N}\in M^{\mathbb N}$ with $X_n\searrow X$, so that
$\hat\EE(X)\geq \EE_\de(X)=\lim_{n\to \infty}\EE(X_n)=\lim_{n\to \infty}\hat\EE(X_n)\geq \hat\EE(X)$.
\end{proof}

\begin{theorem}\label{robustcara}
Suppose that $\FF=\sigma(M)$. For a convex expectation $\EE\colon M\to \R$ which is continuous from above define
\[
  \bar\EE(X):=\sup\left\{\inf_{n\in \N} \EE(X_n)\colon (X_n)_{n\in \N}\in M^{\mathbb N}, X_n\geq X_{n+1}\, (n\in \N), X\geq \inf_{n\in \N}X_n\right\}
 \]
 for all $X\in \LL^\infty(\Om,\FF)$. Then, $\bar\EE$ is the only expectation which is continuous from below on $\LL^\infty(\Om,\FF)$, continuous from above on $M_\de$ and extends $\EE$. Moreover, $\bar\EE$ is convex with the dual representation
\[
 \bar\EE(X)=\sup_{\substack{\mu\in M'\\ \EE^*(\mu)<\infty}}\left(\bar\mu(X)-\EE^*(\mu)\right)=\sup_{\nu\in \ca_+^1(\Omega,\FF)}\left(\nu X-\EE^*(\nu|_M)\right)
\]
 for all $X\in \LL^\infty(\Om,\FF)$. In particular, $\big(\Om,\FF,\bar\EE\big)$ is a convex expectation space.
 \end{theorem}

\begin{proof}
Let $\tilde\EE\colon \LL^\infty(\Om,\FF)\to\mathbb{R}$ be given by
\[
 \tilde\EE(X):=\sup_{\nu\in \ca_+^1(\Om,\FF)}\big(\nu X -\EE^*(\nu|_M)\big)\quad\mbox{for all }X\in \LL^\infty(\Om,\FF).
\]
 By the theorem of Daniell-Stone, it follows  that $\tilde\EE$ is a convex expectation which is continuous from below and extends $\EE$. Moreover, $\tilde\EE$ is continuous from above on $M_\de$. Indeed, let  $(X_n)_{n\in \N}\in M$
 with $X_n\searrow X$ for some $X\in M_\delta$. Define the convex compact set
$\QQ:=\{\mu\in M'\, :\, \EE^*(\mu)\leq \|X_1\|_\infty+\|X\|_\infty\}$
and the mapping
\[
f\colon \QQ\times \N\to \R,\; (\mu, n)\mapsto \mu X_n-\EE^*(\mu),
\]
which is concave on $\QQ$ and convex on $\N$ in the sense of \cite{MR0055678}. Moreover, $f(\, \cdot \, , n)$ is upper semicontinuous for all $n\in \N$.
By Lemma \ref{compactness}, Fan's minimax theorem and the Daniell-Stone theorem, we obtain that
\begin{align*}
\inf_{n\in \N}\tilde\EE(X_n)&=\inf_{n\in \N}\max_{\mu\in \QQ}\big(\mu X_n-\EE^*(\mu)\big)=\max_{\mu\in \QQ}\inf_{n\in \N}\big(\mu X_n-\EE^*(\mu)\big)\\
&\leq \max_{\nu\in \ca_+^1(\Om,\FF)}\inf_{n\in \N}\big(\nu X_n-\EE^*(\nu|_M)\big)=\max_{\nu\in \ca_+^1(\Om,\FF)}\big(\nu X-\EE^*(\nu|_M)\big)\\
&=\tilde\EE(X).
\end{align*}
Hence $\tilde\EE(X)=\inf_{n\in \N}\tilde\EE(X_n)$, so that $\tilde\EE$ is continuous from above on $M_\de$ by Lemma \ref{edelta}.
The claim then follows from Theorem \ref{2.7}, Lemma \ref{edelta} and Choquet's capacibility theorem.
\end{proof}

Although $\bar\EE$ is the only expectation which is continuous from below on $\LL^\infty(\Om,\FF)$, continuous from above on $M_\de$ and extends $\EE$, there may exist infinitely many expectations which extend $\EE$ and are continuous from below as the following example shows.

\begin{example}
 Let $\Om:=[0,1]$, $\FF$ be the Borel $\si$-algebra on $[0,1]$ and
 \[
  \EE(X):=\max_{\om\in \Om}X(\om)=\max_{\mu\in \ca_+^1(\Om,\FF)}\mu X=\max_{\om\in \Om}\de_\om X
 \]
 for all $X\in M:=C([0,1])$, where $\de_\om\in \ca_+^1(\Om,\FF)$ is the Dirac measure $\de_\om(A):=1_A(\om)$ for $\om \in \Om$ and $A\in \FF$.
 Then, $\bar\EE\colon \LL^\infty(\Om,\FF)\to \R$ is given by
\[
 \bar\EE(X)=\sup_{\mu\in \ca_+^1(\Om,\FF)}\mu X=\sup_{\om\in \Om}\de_\om X=\sup_{\om\in \Om}X(\om)\quad \mbox{for all }X\in \LL^\infty(\Om,\FF).
\]
On the other hand, for every $\om_0\in [0,1]$, we have that $\EE_0\colon \LL^\infty(\Om,\FF)\to \R$, given by
\[
 \EE_0(X):=\sup_{\om\in \Om\sm\{\om_0\}}X(\om)=\sup_{\om\in \Om\sm\{\om_0\}}\de_\om X,
\]
is an expectation which extends $\EE$ and is continuous from below.
\end{example}

\begin{corollary}
For a convex pre-expectation $\EE\colon M\to \R$ which is continuous from above, let $\tilde\EE\colon \LL^\infty(\Om,\si(M))\to \R$ be an expectation which is continuous from below and extends $\EE$. Then,
 \begin{enumerate}
  \item[(i)] $\tilde\EE(X)\leq \bar\EE(X)$ for all $X\in \LL^\infty(\Om,\si(M))$,
  \item[(ii)] $\tilde\EE=\bar\EE$ if and only if $\tilde\EE^*(\nu)=\EE^*(\nu|_M)$ for all $\nu\in \ca_+^1(\Om,\si(M))$.
 \end{enumerate}
\end{corollary}

\begin{proof}
 (i) $\tilde\EE\vee\bar\EE$ is an expectation which is continuous from below. Moreover, as $\bar\EE|_{M_\de}=\hat\EE|_{M_\de}$ by Lemma \ref{edelta}, we get that $(\tilde\EE\vee\bar\EE)|_{M_\de}=\bar\EE|_{M_\de}$
 is continuous from above. Hence,  by Theorem \ref{robustcara}, we get that $\tilde\EE\vee\bar\EE=\bar\EE$.\\
(ii) First, we show that $\bar\EE^*(\nu)=\EE^*(\nu|_M)$ for all $\nu\in \ca_+^1(\Om,\si(M))$. Clearly, $\EE^*(\nu|_M)\leq (\bar\EE)^*(\nu)$ for all $\nu\in \ca_+^1(\Om,\si(M))$.
To show the converse inequality, let $\nu\in \ca_+^1(\Om,\FF)$ with $\EE^*(\nu|_M)<\infty$, $X\in \LL^\infty(\Om,\FF)$ and $\ep>0$. By Theorem \ref{robustcara} it follows that $\nu=\overline{(\nu|_M)}$.
Hence, there exists a sequence $(X_n)_{n\in \N}\in M^{\mathbb N}$ with $X_n\geq X_{n+1}$ for all $n\in \N$, $X\geq \inf_{n\in \N}X_n$ and $\nu X \leq \inf_{n\in \N}\nu X_n+\ep$. Further, there exists an $n_0\in \N$ such that $\EE(X_{n_0})-\ep\leq \inf_{n\in \N}\EE(X_n)$.
 Thus, we obtain that
 \begin{align*}
  \nu X-\bar\EE(X)&\leq \inf_{n\in \N}\nu X_n+\ep-\inf_{n\in \N}\bar\EE(X_n)\leq\nu X_{n_0} -\EE(X_{n_0})+2\ep\\
  &\leq \EE^*(\nu|_M)+2\ep.
 \end{align*}
Letting $\ep\searrow 0$, we get that $\nu X-\bar\EE(X)\leq \EE^*(\nu|_M)$ for all $X\in \LL^\infty(\Om,\FF)$.
Hence, if  $\tilde\EE=\bar\EE$ we get that $\tilde\EE^*(\nu)=\bar\EE^*(\nu)=\EE^*(\nu|_M)$ for all $\nu\in \ca_+^1(\Om,\si(M))$.

Now, assume that $\tilde\EE^*(\mu)=\EE^*(\mu|_M)$ for all $\mu\in \ca_+^1(\Om,\si(M))$. By (i) we have that
\[
 \tilde\EE(X)\leq \bar\EE(X)=\sup_{\mu\in \ca_+^1(\Om,\si(M))}(\mu X-\tilde\EE^*(\mu))\leq \tilde\EE(X)
\]
for all $X\in \LL^\infty(\Om,\si(M))$, which shows that $\tilde\EE=\bar\EE$.
\end{proof}


\section{A robust version of Kolmogorov's extension theorem}\label{sec4}
In this section, we apply the previous results to a Kolmogorov type setting. That is, given a family of finite dimensional marginal expectations, we want to find an expectation
with these marginals. Again, we will distinguish between the finitely additive case and the countably additive case. Finally, we will state a robust version of Kolmogorov's extension theorem.

Throughout, let $I\neq \es$ be an index set, $\HH:=\{J\subset I\colon \# J\in \N\}$ the set of all finite nonempty subsets of $I$ and $S$ a Polish space with the Borel $\si$-algebra $\BB$.
For each $J\in \HH$ let $M_J\subset \LL^\infty(S^J,\BB^J)$ be a linear subspace with $1\in M_J$, where $\BB^J$ is the product $\si$-algebra on $S^J$.
As before, $M_J$ is always endowed with the $\|\cdot\|_\infty$-norm and on $(M_J)'$ we consider the weak$^*$-topology.
Throughout this section, we assume that
\[M_K\circ \pr_{JK}:=\left\{f\circ \pr_{JK}\colon f\in M_K\right\}\subset M_J\]
for all $J,K\in \HH$ with $K\subset J$ where
$\pr_{JK}\colon S^J\to S^K$, $(x_i)_{i\in J}\mapsto (x_i)_{i\in K}$ and $\pr_J:=\pr_{IJ}$.
For $\mu_J\in (M_J)'$ we define
\[
 (\mu_J\circ \pr_{JK}^{-1})f:=\mu_J(f\circ \pr_{JK})\quad \mbox{for all }f\in M_K
 \]
so that \[(M_J)'\circ\pr_{JK}^{-1}:=\{\mu_J \circ \pr_{JK}^{-1}\colon \mu_J\in (M_J)'\}\subset (M_K)'.\]
Notice that the linear mapping $(M_J)'\to (M_K)',\;\mu_J \mapsto \mu_J\circ \pr_{JK}^{-1}$ is continuous.

\begin{remark}
 Typical examples for the family $(M_J)_{J\in \HH}$ are:
 \begin{enumerate}
  \item[(i)] the space $\LL^\infty(S^J):=\LL^\infty(S^J,\BB^J)$ of all bounded $\BB^J$-$\BB(\R)$-measurable functions, where $\BB^J$ denotes the product $\si$-algebra on $S^J$,
  \item[(ii)] the space $M_J:=C_b(S^J)$ of all continuous bounded functions $S^J\to \R$, where $S^J$ is endowed with the product topology.
 \end{enumerate}
\end{remark}

In \cite{MR2143645}, Peng defines a consistency condition for nonlinear expectations and proves an extension to the subspace
\[
 M:=\{f\circ \pr_J\colon J\in \HH,f\in \LL^\infty(S^J,\BB^J)\}
\]
of $\LL^\infty(S^I,\BB^I)$. We will use the same notion of consistency as Peng and apply the extension results from the previous sections to obtain an extension to $\LL^\infty(S^I,\BB^I)$.

\begin{definition}
For all $J\in \HH$ let $\EE_J\colon M_J\to \R$ be a pre-expectation. Then the family $(\EE_J)_{J\in \HH}$ is
\textit{consistent} if for all $J,K\in \HH$ with $K\subset J$
\[
\EE_K(f)=\EE_J(f\circ \pr_{JK})\quad \mbox{for all }f\in M_K.
\]
A family $(\QQ_J)_{J\in \HH}$ of subsets $\QQ_J\subset (M_J)'$ is
\textit{consistent} if for all $J,K\in \HH$ with $K\subset J$
\[
\QQ_J\circ \pr_{JK}^{-1}=\QQ_K.
\]
\end{definition}

\begin{lemma}\label{equivalence}
 For every $J\in \HH$ let $\EE_J\colon M_J\to \R$ be a sublinear pre-expectation and
 \[
  \QQ_J:=\{\mu_J\in (M_J)'\colon \mu_J f\leq \EE_J(f)\mbox{ for all } f\in M_J \}.
  \]
Then, the family $(\EE_J)_{J\in \HH}$ is consistent if and only if the family $(\QQ_J)_{J\in \HH}$ is consistent.
\end{lemma}

\begin{proof}
Suppose that  $(\EE_J)_{J\in \HH}$ is consistent. Then, by Lemma \ref{transfo}, we obtain that the family $(Q_J)_{J\in \HH}$ is consistent, as well.

Now suppose that the family $(\QQ_J)_{J\in \HH}$ is consistent and let $J,K\in \HH$ with $K\subset J$. Then, by Lemma \ref{compactness}, we get that
\begin{align*}
 \EE_K(f)&=\max_{\mu_K\in \QQ_K}\mu_Kf=\max_{\mu_K\in \QQ_J\circ \pr_{JK}^{-1}} \mu_K f\\
 &=\max_{\mu_J\in \QQ_J}\mu_J(f\circ \pr_{JK})=\EE_J(f\circ \pr_{JK})
\end{align*}
for all $f\in M_K$.
\end{proof}

In the following, we denote by $\BB^I$ the product $\si$-algebra, which is generated by the sets of the form $\pr_J^{-1}(B)$, where $J\in \HH$ and $B\in \BB^J$.

\begin{proposition}\label{riskkolmogorov}
 Let $(\EE_J)_{J\in \HH}$ be a consistent family of pre-expectations $\EE_J\colon M_J\to\R$. Then, there exists an expectation $\hat\EE\colon \LL^\infty(S^I,\BB^I)\to \R$
 such that
\[
\hat\EE(f\circ \pr_J)=\EE_J(f)\quad\mbox{for all }J\in \HH\mbox{ and all }f\in M_J.
\]
If the pre-expectations $\EE_J$ are convex or sublinear for all $J\in \HH$, then $\hat\EE$ is convex or sublinear, respectively.
\end{proposition}

\begin{proof}
 Let $M:=\{f\circ \pr_J\colon f\in M_J, \; J\in \HH\}$. Then $M$ is a linear subspace of $\LL^\infty(S^I,\BB^I)$ with $1\in M$.
 For every $J\in \HH$ and $f\in M_J$ let $\EE(f\circ \pr_J):=\EE_J(f)$. Since the family $(\EE_J)_{J\in \HH}$ is consistent,
 the functional $\EE\colon M\to\R$ is well-defined. Moreover, $\EE\colon M\to \R$ is a pre-expectation on $M$.
The assertion then follows from Proposition \ref{bacara}.
\end{proof}

Note that Proposition \ref{riskkolmogorov} still holds without the assumption that $S$ is a Polish space. In fact, $S$ could be an arbitrary state space. If $\EE_J$ is linear for all $J\in \HH$, by Remark \ref{2.6}, we obtain that
\[
\hat\EE(f)=\sup_{\nu\in \PP} \nu f
\]
where $\PP:=\{\nu\in \ba_+^1(S^I,\BB^I)\colon \nu\circ \pr_J=\EE_J\mbox{ for all } J\in \HH\}$.
For all $J\in \HH$ let $[(M_J)']_+^1$ denote the set of all linear pre-expectations $\mu_J\colon M_J\to \R$.

\begin{corollary}\label{barobustkolmogorov}
Suppose that $\QQ_J\subset [(M_J)']_+^1$ is convex and compact for all $J\in \HH$, and
the family $(\QQ_J)_{J\in \HH}$ is consistent.
Then, there exists a convex and compact set $\QQ\subset \ba_+^1(S^I,\BB^I)$ such that $\QQ\circ \pr_J^{-1}=\QQ_J$ for all $J\in \HH$, where $(\mu\circ \pr_J^{-1})f:=\mu(f\circ \pr_J)$
for all $\mu\in \ba(S^I,\BB^I)$ and $f\in M_J$.
\end{corollary}

\begin{proof}
For each $J\in\HH$ define the sublinear pre-expectation \[\EE_J(f):=\max_{\mu_J\in \QQ_J} \mu_J f\quad \mbox{for all }f\in M_J.\]
Since $\QQ_J\subset [(M_J)']_+^1$ is convex and compact, the separation theorem of Hahn-Banach implies that
\begin{equation}\label{hahnbanach}
 \QQ_J=\{\mu_J\in M_J'\colon  \mu_J f\leq \EE_J(f)\mbox{ for all } f\in M_J\}.
\end{equation}
Hence, by Lemma \ref{equivalence}, the family $(\EE_J)_{J\in \HH}$ is consistent, and by Theorem \ref{riskkolmogorov} there exists a sublinear
expectation $\hat\EE\colon \LL^\infty(S^I,\BB^I)\to \R$ such that
\begin{equation}\label{consistency}
 \hat\EE(f\circ \pr_J)=\EE_J(f)\quad\mbox{for all }J\in \HH \mbox{ and all }f\in M_J.
\end{equation}
 By Lemma \ref{compactness}, $\hat\EE(f)=\max_{\mu\in\QQ}\mu f$ for all $f\in \LL^\infty(S^I,\BB^I)$ where
\[
 \QQ:=\{\mu\in \ba_+^1(S^I,\BB^I)\colon  \mu f\leq \hat \EE(f)\mbox{ for all } f\in \LL^\infty(S^I,\BB^I)\}.
\]
By Lemma \ref{transfo}, we thus obtain that $\QQ\circ \pr_J^{-1}=\QQ_J$ for all $J\in \HH$.
\end{proof}

\begin{theorem}\label{extension}
For every $J\in \HH$, let $M_J$ be a Riesz subspace of $\LL^\infty(S^J,\BB^J)$ with $\si (M_J)=\BB^J$ and $\EE_J\colon M_J\to \R$ be a convex pre-expectation which is continuous from above. Assume that the family $(\EE_J)_{J\in \HH}$ is consistent.
Then, there exists exactly one expectation $\bar\EE\colon \LL^\infty(S^I,\BB^I)\to \R$ which is continuous from below  on $\LL^\infty(S^I,\BB^I)$ and continuous from above on $M_\de$, where $M:=\{f\circ\pr_J\colon f\in M_J, \, J\in \HH\}$, such that
\[
\EE_J(f)=\bar\EE(f\circ \pr_J)=\sup_{\nu\in \ca_+^1(S^I,\BB^I)}\big(\nu (f\circ \pr_J)-\bar\EE^*(\nu)\big)
\]
for all $J\in \HH$ and all $f\in M_J$. Moreover, $\bar\EE$ is convex and if the pre-expectations $\EE_J$ are sublinear or linear for all $J\in \HH$, then $\bar\EE$ is sublinear or linear, respectively.
\end{theorem}

\begin{proof}
Define $\EE(f\circ \pr_J):=\EE_J(f)$ for all $f\in M_J$ and all $J\in \HH$. Since the family $(\EE_J)_{J\in \HH}$ is consistent, $\EE\colon M\to \R$ defines a convex pre-expectation on $M$.
Let $\mu\in M'$ with $\EE^*(\mu)<\infty$. We will first show that $\mu\colon M\to \R$ is continuous from above. Let $\mu_J:=\mu\circ \pr_J^{-1}\in M_J'$ for all $J\in \HH$. Then,
$(\EE_J)^*(\mu_J)\leq \EE^*(\mu)<\infty$ and by Lemma \ref{contabove} $\mu_J\colon M_J\to \R$ is continuous from above. By the theorem of Daniell-Stone, there exists a unique $\nu_J\in \ca_+^1(S^J,\BB^J)$ with $\nu_J|_{M_J}=\mu_J$ for all $J\in \HH$. As
\[
 \mu_K=\mu_J\circ \pr_{JK}^{-1}=(\nu_J\circ\pr_{JK}^{-1})|_{M_K},
\]
we thus obtain that $\nu_K=\nu_J\circ \pr_{JK}^{-1}$ for all $J,K\in \HH$ with $K\subset J$. By Kolmogorov's extension theorem, there exists a unique $\nu\in \ca_+^1(S^I,\BB^I)$ with $\nu(f\circ \pr_J)=\nu_J f$ for all $f\in \LL^\infty(S^J)$ and $J\in \HH$. Hence, we get that $\nu|_M=\mu$, and therefore $\mu\colon M\to \R$ is continuous from above, as well. By Lemma \ref{contabove} we thus obtain that $\EE\colon M\to \R$ is continuous from above.

Next, we will show that $\BB^I\subset \si(M)$. Let $J\in \HH$ and $B_J\in \BB^J$. Then, $\BB^J\in \si(M_J)$ and therefore $\pr_J^{-1}(B_J)\in \si(M_J\circ \pr_J)\subset \si(M)$.

Finally, since $M$ is a Riesz subspace of $\LL^\infty(S^I,\BB^I)$ with $1\in M$ and $\BB^I\subset \si(M)$, the assertion  follows from Theorem \ref{robustcara}.
\end{proof}

\begin{theorem}\label{kolmogorov}
For every $J\in \HH$, let $M_J$ be a Riesz subspace of $\LL^\infty(S^J,\BB^J)$ with $\si(M_J)=\BB^J$ and $\QQ_J\subset [(M_J)']_+^1$ be convex and compact. Assume that for all $J\in \HH$ every $\mu_J\in \QQ_J$ is continuous from above and that the family $(\QQ_J)_{J\in \HH}$ is consistent.
Then, there exists a set $\QQ\subset \ca_+^1(S^I,\BB^I)$ with
\[
\QQ_J=\QQ\circ \pr_J^{-1}\quad \mbox{for all }J\in \HH.
\]
\end{theorem}

\begin{proof}
Let
\[
 \QQ:=\{\mu\in \ca_+^1(S^I,\BB^I)\colon (\mu\circ \pr_J^{-1})|_{M_J}\in \QQ_J\mbox{ for all }J\in \HH\}.
\]
Then, $\QQ\circ \pr_J^{-1}\subset \QQ_J$ for all $J\in \HH$. In order to show the other implication, let $J_0\in \HH$ and $\mu_{J_0}\in \QQ_{J_0}$ be fixed. By Corollary \ref{barobustkolmogorov}, there exists a $\mu\in \ba_+(S^I,\BB^I)$ with $\mu\circ\pr_{J_0}^{-1}=\mu_{J_0}$ and $\mu\circ \pr_J^{-1}\in \QQ_J$ for all $J\in \HH$. Let $\mu_J:=\mu\circ \pr_J^{-1}\in \QQ_J$ for all $J\in \HH\sm \{J_0\}$. Then, the family $(\mu_J)_{J\in \HH}$ is consistent and $\mu_J\in \QQ_J$ is continuous from above for all $J\in \HH$. By the theorem of Daniell-Stone, there exists a unique $\nu_J\in \ca_+^1(S^J,\BB^J)$ with $\nu_J|_{M_J}=\mu_J$ for all $J\in \HH$. Since
\[
 \mu_K=\mu_J\circ \pr_{JK}^{-1}=(\nu_J\circ\pr_{JK}^{-1})|_{M_K},
\]
we obtain that $\nu_K=\nu_J\circ \pr_{JK}^{-1}$ for all $J,K\in \HH$ with $K\subset J$. Hence, by Kolmogorov's extension theorem, there exists a
unique $\nu\in \ca_+^1(S^I,\BB^I)$ with $\nu\circ \pr_J^{-1}=\nu_J$ for all $J\in \HH$. Hence, $\nu\in \QQ$ and satisfies $(\nu\circ\pr_{J_0}^{-1})|_{M_{J_0}}=\mu_{J_0}$.
\end{proof}

\begin{example}
 Let $S:=\{0,1\}$ be endowed with the topology $2^S$. Then, $S$ is a Polish space with Borel-$\si$-algebra $2^S$. Let $\HH:=\{J\subset \N\colon \#J\in \N\}$ be the set of all finite nonempty subsets of $\N$. Then, for all $J\in \HH$ we have that $\#S^J<\infty$ and therefore the product $\si$-algebra $\BB^J$ is the power set $2^{S^J}$ and $\LL^\infty(S^J,\BB^J)$ consists of all functions $S^J\to \R$. Let $M:=\{f\circ \pr_J\colon  J\in \HH,\, f\colon S^J\to \R \}$. For $y\in S^\N$ let $\de_y\in \ca_+^1(S^\N,\BB^\N)$ denote the Dirac measure given by $\de_y(B)=1_B(y)$ for $B\in \BB^\N$.
 \begin{enumerate}
  \item[a)] For $n\in \N$ let $S^n:=S^{\{1,\ldots, n\}}$ and $\pr_n:=\pr_{\{1,\ldots, n\}}$. Let $y\in S^\N$ and $\EE\colon M\to \R$ be given by $\EE(g):=\de_y g$ for $g\in M$. Let $f:=1_{S^\N\sm \{y\}}\in \LL^\infty(S^\N,\BB^\N)$. For $n\in \N$ let
  \begin{align*}
   B_n&:=\pr_n^{-1}\big(\{(y_1,\ldots, y_n)\}\big)\\&=\big\{x\in S^\N\colon x_i=y_i \mbox{ for all }i\in \{1,\ldots, n\}\big\}\in \BB^\N
  \end{align*}
  and $g_n:=1_{S^\N\sm B_n}=1-1_{B_n}\in M$. Then, we have that $g_n\nearrow f$ as $n\to \infty$, i.e. $f\in M_\si$. In fact, by definition we have that $g_n(y)=0=f(y)$ for all $n\in \N$. Let $x\in S^\N\sm \{y\}$. Then, there exists some $i\in \N$ with $x_i\neq y_i$ and therefore $g_n(x)\nearrow 1=f(x)$ as $n\to \infty$. As $y\in B_n$ we have that $\hat\EE(g_n)=\EE(g_n)=\de_y g_n=0$ for all $n\in \N$. Let $g\in M$ with $g\geq f$. Then, we have that $g(x)\geq f(x)=1$ for all $x\in S^\N\sm\{y\}$. On the other hand, there exists some $J\in \HH$ and some $h\colon S^J\to \R$ such that $g=h\circ \pr_J$. As $\#S=2>1$, there exists some $x\in S^\N\sm \{y\}$ with $\pr_J(x)=\pr_J(y)$ and therefore
  \[
   g(y)=h(\pr_J(y))=h(\pr_J(x))=g(x)\geq 1.
  \]
  This shows $g(x)\geq 1$ for all $x\in S^\N$. As $1\geq f$ and $1\in M$ we obtain that
  \[
  \hat\EE(f)=1\neq 0=\lim_{n\to \infty}\hat\EE(g_n).
  \]
  This shows that in general $\hat\EE$ is not continuous from below, not even on $M_\si$.
  \item[b)] In general, we may not expect that the set $\QQ\subset \ca_+^1(S^I,\BB^I)$ from Theorem \ref{kolmogorov} to be compact, not even if $\QQ_J$ is convex for all $J\in \HH$. In fact, let $\QQ:=\ca_+^1(S^\N,\BB^\N)$. Then, $\QQ$ is convex and $\si(\ca_+^1(S^\N,\BB^\N),\LL^\infty(S^\N,\BB^\N))$-closed. But, $\QQ$ is not a compact subset of $\ba_+^1(S^\N,\BB^\N)$ as $\ca_+^1(S^\N,\BB^\N)\neq \ba_+^1(S^\N,\BB^\N)$. On the other hand, we have that $\QQ\circ \pr_J^{-1}=\ba_+^1(S^J,\BB^J)$ is a convex and compact for all $J\in \HH$.
  \item[c)] Let $\QQ:=\ca_+^1(S^\N,\BB^\N)$ and $\PP:=\{\nu\in \ca_+^1(S^\N,\BB^\N) : \nu(\{y\})=0\}$ for some $y\in S^\N$. Then, we have that $\PP$ and $\QQ$ are both convex and $\si(\ca_+^1(S^\N,\BB^\N),\LL^\infty(S^\N,\BB^\N))$-closed. Let $J\in \HH$. Since $\#S=2>1$, there exists some $x_J\in S^\N\setminus \{y\}$ with $\pr_J(x_J)=\pr_J(y)$ and therefore
  \[
   \PP\circ \pr_J^{-1}=\ba_+^1(S^J,\BB^J)=\QQ\circ \pr_J^{-1}
  \]
  is compact. On the other hand we have that $\PP\neq \QQ$. This shows that in Theorem \ref{kolmogorov} no uniqueness can be obtained. A similar example shows that also in Theorem \ref{barobustkolmogorov} uniqueness cannnot be obtained.
 \end{enumerate}
\end{example}

\begin{example}
 Let $0<\underline\si\leq \overline\si$, $\underline\mu\leq \overline\mu$ and $T>0$. Let $n\in \N$, $\si\in [\underline\si,\overline\si]^n$, $\mu\in [\underline\mu,\overline\mu]^n$ and $0\leq t_1<\ldots<t_n\leq T$. For $J:=\{t_1,\ldots,t_n\}$ and $f\in C_b(\R^n)$ let
 \[
  \E_J^{\mu,\si}(f):=\int_{\R^n} f(x_1,x_1+x_2,\ldots ,x_1+\ldots+ x_n)\, {\rm d}N^{\mu,\si}_J(x_1,\ldots, x_n),
 \]
for the product of normal distributions
\[
 N^{\mu,\si}_J:=\bigotimes_{k=1}^nN\big(\mu_k(t_k-t_{k-1}),\si_k^2(t_k-t_{k-1})\big)
\]
with $t_0:=0$ and $N(0,0):=\de_0$. Moreover, let
\[
 \EE_J(f):=\sup_{\mu\in [\underline\mu,\overline\mu]^n,\si\in [\underline\si,\overline\si]^n}\E_J^{\mu,\si}(f)
\]
for all $f\in C_b(\R^n)$. We equip $\ca_+^1(\R^n,\BB(\R)^n)$ with the $C_b(\R^n)$-weak topology. Then, the mapping
\[
 \R^n\times [0,\infty)^n\to \ca_+^1(\R^n,\BB(\R)^n),\quad (\mu,\si)\mapsto N^{\mu,\si}_J
\]
is continuous by L\'evy's continuity theorem or by direct computation (note that it suffices to verify sequential continuity as $\R^n\times [0,\infty)^n$ is a metric space). Let $s\colon \R^n\to \R^n$ be given by
\[
 s(x_1,\ldots, x_n):=(x_1,x_1+x_2,\ldots, x_1+\ldots+x_n)\quad \mbox{for all }x_1,\ldots, x_n\in \R.
\]
As $s\colon \R^n\to \R^n$ is continuous, the mapping
\[
 \ca_+^1(\R^n,\BB(\R^n))\to \ca_+^1(\R^n,\BB(\R^n)), \quad \nu\mapsto \nu\circ s^{-1}
\]
is continuous and therefore, the mapping
\[
 \R^n\times [0,\infty)^n\to \ca_+^1(\R^n,\BB(\R)^n),\quad (\mu,\si)\mapsto \E_J^{\mu,\si}
\]
is continuous. As $[\underline\mu,\overline\mu]^n\times [\underline\si,\overline \si]^n\subset \R^n\times [0,\infty)^n$ is compact, we thus obtain that the family
\[
 \QQ_J:=\{\E_J^{\mu,\si}\colon \mu\in [\underline\mu,\overline\mu]^n,\si\in [\underline\si,\overline\si]^n\}\subset C_b(\R^n)'
\]
is compact. Since $\QQ_J$ is convex and compact, we get that $\EE_J$ is continuous from above with
\[
 \QQ_J=\{\mu_J\in C_b(\R^n)'\colon \mu_Jf\leq \EE_J(f)\mbox{ for all }f\in C_b(\R^n)\}
\]
for each $J\in \HH$. As the family $(\QQ_J)_{J\in \HH}$ is consistent, we thus get that the family $(\EE_J)_{J\in \HH}$ is consistent by Lemma \ref{equivalence}. Hence, we may apply Theorem \ref{extension} and Theorem \ref{kolmogorov} and obtain an expectation $\bar\EE\colon \LL^\infty(\R^{[0,T]},\BB(\R)^{[0,T]})\to \R$ and a set $\QQ\subset \ca_+^1(\R^{[0,T]},\BB(\R)^{[0,T]})$ with $\EE\circ \pr_J^{-1}=\EE_J$ and $\QQ\circ\pr_J^{-1}=\QQ_J$ for all $J\in \HH$. However, this is not the $G$-expectation introduced by Peng \cite{PengG},\cite{MR2474349}, see also Example \ref{controlapp}.
\end{example}


\section{Application to nonlinear kernels}\label{sec5}

Let $(S,\BB)$ be a measurable space. We will apply the nonlinear Kolmogorov theorem to nonlinear kernels. We follow the definition of a monetary risk kernel by
F\"ollmer and Kl\"uppelberg \cite{follmer2014spatial} and define nonlinear kernels in an analogous way. We will use the results from the previous section to extend these nonlinear kernels.
Throughout this subsection, let $M,N\subset \LL^\infty(S,\BB)$ with $\R\subset M$ and $\R\subset N$.

\begin{definition}
A \textit{(nonlinear) pre-kernel} from $M$ to $N$ is a function $\EE\colon S \times M\to \R$ such that
\begin{enumerate}
 \item[(i)] for each $x\in S$, the function $M\to \R,\; f\mapsto \EE(x,f)$ is a (nonlinear) pre-expectation,
 \item[(ii)] for every $f\in M$, the function $\EE(\, \cdot\, , f)\in N$.
\end{enumerate}
We say that a pre-kernel $\EE$ from $M$ to $N$ is convex, sublinear, continuous from above, or continuous from below,
if for every  $x\in S$ the function $\EE(x,\, \cdot \, )$ is convex, sublinear, continuous from above, or continuous from below, respectively.
\end{definition}

For two pre-kernels $\EE_0,\EE_1$ from $M$ to $M$ we write
\[
 (\EE_0\EE_1)(x,f):=\EE_0(x, \EE_1(\, \cdot\, ,f))
\]
for all $x\in S$ and all $f\in M$. Then, one easily checks that $\EE_0\EE_1$ is a pre-kernel, again.

\begin{definition}
We say that a family $(\EE_{s,t})_{0\leq s<t< \infty}$ of pre-kernels from $M$ to $M$ fulfills the
\textit{Chapman-Kolmogorov equations}, if
$
 \EE_{s,u}=\EE_{s,t}\EE_{t,u}
$
for all $0\leq s<t<u<\infty$.
\end{definition}

\begin{example}\label{markchain}
 Let $S$ be a finite state space and $\BB:=2^S$, so that $\LL^\infty(S,\BB)=\R^S$. Let \[\PP\colon \LL^\infty(S,\BB)\to \LL^\infty(S,\BB)\quad and \quad \mu_0\colon \LL^\infty(S,\BB)\to \R\] be convex and therefore continuous, constant preserving, i.e.~$\PP(\al)=\al$ and $\mu_0(\al)=\al$ for all $\al\in \R$, and monotone, i.e.~$\PP(f)\leq \PP(g)$
and  $\mu_0(f)\leq \mu_0(g)$ for all $f,g\in \LL^\infty(S,\BB)$ with $f\le g$.
 For every $k,l\in \N_0$ with $k< l$, we define
 \[
  \EE_{k,l}(\, \cdot\, , f):=\PP^{l-k}(f)\quad\mbox{for all }f\in \LL^\infty(S,\BB).
 \]
 Then, $\EE_{k,l}\colon S\times \LL^\infty(S,\BB)\to \R$ defines a convex kernel from $\LL^\infty(S,\BB)$ to $\LL^\infty(S,\BB)$ for all $k,l\in \N_0$ with $k< l$. Let $\HH:=\{J\subset \N_0\colon \# J\in \N\}$ be the set of all finite, nonempty subsets of $\N_0$. For $k\in\N_0$ we define
\[
 \EE_{\{k\}}(f)
 :=\mu_0(\PP^k(f))\quad \mbox{for all }f\in \LL^\infty(S,\BB),
\]
where ${\PP}^0$ is the identity. For $n\in \N$, $k_1,\ldots, k_{n+1}\in \N_0$ with $k_1<\ldots<k_{n+1}$ and $f\in \LL^\infty(S^{n+1},\BB^{n+1})$ we now define recursively
\[
 \EE_{\{k_1,\ldots, k_{n+1}\}}(f):=\EE_{\{k_1,\ldots, k_n\}}(g)
\]
where $g(x_1,\ldots, x_n):=\EE_{k_n,k_{n+1}}(x_n,f(x_1,\ldots, x_n,\, \cdot \, ))$ for all $x_1,\ldots, x_n\in S$.
Then, $\EE_J\colon \LL^\infty(S^J,\BB^J)\to \R$ is a convex expectation which is continuous from above for all $J\in \HH$.
Since the family $(\EE_{k,l})_{0\leq k<l}$ fulfills the Chapman-Kolmogorov equations, we obtain that the family $(\EE_J)_{J\in \HH}$ is consistent. Hence, by Theorem \ref{extension} there exists an expectation $\EE\colon \LL^\infty(S^{\N_0},\BB^{\N_0})\to \R$ which is continuous from below and satisfies
\[
 \EE((f\circ\pr_k)(g\circ\pr_l))=\mu_0(\PP^k(f\PP^{l-k}(g)))
\]
for all $f,g\in \LL^\infty(S,\BB)$ and $k,l\in \N_0$ with $k<l$. Hence, $(\pr_k)_{k\in \N_0}$ can be viewed as a convex Markov chain on $(S^{\N_0},\BB^{\N_0},\EE)$. If $\PP$ is sublinear, the set
\[
\big\{\mu\in \R^{S\times S}\colon \mu f\leq \PP(f)\; \text{for all}\; f\in \R^S \big\}
\]
induces a Markov-set chain, see Hartfiel \cite{MR1725607}.
\end{example}

In the following, let $S$ be a Polish space with metric $d$ and Borel $\si$-algebra $\BB$. We denote by $\BUC(S)$ the space of all bounded and uniformly continuous functions w.r.t.~the metric $d$. On general state spaces the measurability of $g$
in the above example is non-trivial. In the following we will therefore consider pre-kernels from $C_b(S)$ to $C_b(S)$.

\begin{remark}\label{remnonlinearkern}\
 \begin{enumerate}
  \item[a)] We have that $\BUC(S)_\si=\LSC_b(S)$, where $\LSC_b(S)$ denotes the space of all bounded lower semicontinuous functions $S\to \R$. In fact, the implication $\BUC(S)_\si\subset \LSC_b(S)$ is obvious. To show the inverse implication, let $f\in \LSC_b(S)$, where w.l.o.g.~we may assume that $f\geq 0$ (otherwise consider $f+\|f\|_\infty$). For $k,n\in \N_0$ let $U_k^n:=\{x\in S\colon f(x)> k 2^{-n}\}$. As $U_n^k$ is open, we have that $ k 2^{-n} 1_{U_k^n}\in \BUC(S)_\si$ for all $k,n\in \N_0$. Note that $n\big( d(x,U^c)\wedge \frac{1}{n}\big)\nearrow 1_U(x)$ as $n\to \infty$ for all $x\in S$ and any open set $U\subset S$. Finally, for all $n\in \N_0$ let
 \[
  f_n:=\sup_{k\in \N_0} k 2^{-n} 1_{U_k^n}.
 \]
 Then, we have that $f_n\in \BUC(S)_\si$ with $f_n\leq f_{n+1}\leq f$ and $\|f-f_n\|_\infty\leq 2^{-n}$ for all $n\in \N_0$. In particular, $f_n\nearrow f$ as $n\to \infty$, and therefore, $f\in \BUC(S)_\si$. As $\BUC(S)$ is a vector space, we thus obtain that $\BUC(S)_\de=\USC_b(S)$, where $\USC_b(S)$ denotes the space of all bounded upper semicontinuous functions $S\to \R$.
 \item[b)] Let  $M$ be a dense subspace of $\BUC(S)$ with $1\in M$, and
$\EE$ a convex pre-kernel from $M$ to $M$. Then, as $\EE$ is $1$-Lipschitz, there exists exactly one convex pre-kernel $\hat\EE$ from $\BUC(S)$ to $\BUC(S)$ with $\hat\EE|_{M}=\EE$.
 \item[c)] Let $\EE$ be a convex pre-kernel from $\BUC(S)$ to $C_b(S)$, which is continuous from above. Then, there exists exactly one convex pre-kernel $\hat\EE$ from $C_b(S)$ to $C_b(S)$, which is continuous from above and satisfies $\hat\EE|_{\BUC(S)}=\EE$. Indeed, by part a) and Lemma \ref{edelta}, there exists a convex kernel $\hat\EE$ from $C_b(S)$ to $\LL^\infty(S)$, which is continuous from above and extends $\EE$. Since $C_b(S)$ is a vector space, it follows that $\hat\EE$ is continuous from below, as well.  By part a), for $f\in C_b(S)$ there exist sequences $(f_n)_{n\in \N}$  and  $(g_n)_{n\in \N}$ in $\BUC(S)$ with $f_n\searrow f$ and $g_n\nearrow f$ as $n\to \infty$. Therefore,
 \[
  \inf_{n\in \N}\EE(f_n)=\hat\EE(f)=\sup_{n\in \N}\EE(g_n),
 \]
 which shows that $\hat\EE(f)\in C_b(S)$.
 \item[d)] Let $(\EE_{s,t})_{0\leq s<t<\infty}$ be a family of convex pre-kernels from $C_b(S)$ to $C_b(S)$, which are continuous from above. Moreover, let $M\subset \BUC(S)$ be a dense subspace of $\BUC(S)$ with $1\in M$. Then, by the uniqueness obtained in part b) and c), the following statements are equivalent:
 \begin{enumerate}
 \item[(i)] $(\EE_{s,t})_{0\leq s<t< \infty}$ satisfies the Chapman-Kolmogorov equations,
 \item[(ii)] $\EE_{s,u}(f)=\EE_{s,t}(\EE_{t,u}(f))$ for all $f\in \BUC(S)$ and $0\leq s<t<u<\infty$,
 \item[(iii)] $\EE_{s,u}(f)=\EE_{s,t}(\EE_{t,u}(f))$ for all $f\in M$ and $0\leq s<t<u<\infty$.
 \end{enumerate}
 Therefore, the extension of convex kernels from $\BUC(S)$ to $\BUC(S)$ or from $M$ to $M$, which are continuous from above, are included in the extension of pre-kernels from $C_b(S)$ to $C_b(S)$, which are continuous from above.
 \end{enumerate}
\end{remark}

\begin{proposition}\label{parameterdep}
Let $\EE$ be a convex pre-kernel from $C_b(S)$ to $C_b(S)$, which is continuous from above.
Then, for every Polish space $T$ the parameter dependent version
 \[
  \EE_{T}\colon (S\times T)\times C_b(S\times T)\to \R,\quad ((x,y),f)\mapsto \EE(x,f(\, \cdot \,,y))
 \]
 is a convex pre-kernel from $C_b(S\times T)$ to $C_b(S\times T)$, which is continuous from above.
\end{proposition}

\begin{proof}
 First note that for all $(x,y)\in S\times T$, the function \[C_b(S\times T)\to \R,\quad f\mapsto \EE_T\big((x,y),f\big)\] is a convex pre-expectation on $C_b(S\times T)$, which is continuous from above.
 Let $f\in C_b(S\times T)$ be uniformly continuous w.r.t.~the metric $d(x,x^\prime)+d_T(y,y^\prime)$
 for $(x,y),(x',y')\in S\times T$, where $d_T$ is a metric on $T$. Fix $\ep>0$, and $(x_0,y_0)\in S\times T$. Then, there exists a $\delta>0$ such that
 \[
  |f(x,y_0)-f(x,y)|\leq\frac{\ep}{2}
 \]
 for all $(x,y)\in S\times T$ with $d_T(y,y_0)<\de$, and
 \[
  |\EE(x_0,f_{y_0})-\EE(x,f_{y_0})|\leq \frac{\ep}{2}
 \]
 for all $x\in S$ with $d(x,x_0)\leq\de$. Here, $f_{y}\colon S\to \R$ is defined as $x\mapsto f(x,y)$ for all $y\in T$.
 Then, $\|f_y-f_{y_0}\|_\infty\leq \frac{\ep}{2}$ for all $y\in T$ with $d_T(y,y_0)\leq \de$. Therefore,
\begin{align*}
|\EE_T((x_0,y_0),f)-\EE_T((x,y),f)|&=|\EE(x_0,f_{y_0})-\EE(x,f_y)|\\
&\leq |\EE(x_0,f_{y_0})-\EE(x,f_{y_0})|+|\EE(x,f_{y_0})-\EE(x,f_y)|\\
&\leq |\EE(x_0,f_{y_0})-\EE(x,f_{y_0})|+\|f_{y_0}-f_y\|_\infty\\
&<\frac{\ep}{2}+\frac{\ep}{2}=\ep
\end{align*}
for all $(x,y)\in S$ with $d(x,x_0)\leq\de$ and $d_T(y,y_0)\leq \de$. By Remark \ref{remnonlinearkern} c), we obtain the assertion.
\end{proof}

The following result allows to generalize the construction from Example \ref{markchain} to general Markov processes.

\begin{theorem}\label{chapman}
 Let $(\EE_{s,t})_{0\leq s<t<\infty}$ be a family of convex kernels from $C_b(S)$ to $C_b(S)$, which fulfills the
 Chapman-Kolmogorov equations and $\EE_0\colon C_b(S)\to \R$ be a convex pre-expectation. Further, assume that $\EE_0$ is continuous from above and
 $\EE_{s,t}$ is continuous from above for all $0\leq s<t\leq T$. Then, there exists a nonlinear expectation space $(\Om,\FF,\EE)$ and a stochastic process $(X_t)_{t\geq 0}$ of random variables $\Om\to S$, which satisfies
 \begin{enumerate}
  \item[(i)] $\EE\big(f(X_0)\big)=\EE_0(f)$ for all $f\in C_b(S)$,
  \item[(ii)] For all $0\leq s<t$, $n\in \N$, $0\leq t_1<\ldots <t_n\leq s$ and $f\in C_b(S^{n+1})$ we have that
  \[
   \EE\big(f(X_{t_1},\ldots,X_{t_n},X_t)\big)=\EE\left(\EE_{s,t}\big(X_s,f(X_{t_1},\ldots,X_{t_n},\, \cdot\,)\big)\right).
   \]
 \end{enumerate}

\end{theorem}

\begin{proof}
Let $\EE_{0,0}(\, \cdot\, , f):=f$ for all $f\in C_b(S)$ and $\HH:=\{J\subset [0,\infty)\colon \# J\in \N \}$ be the set of all finite, nonempty subsets of $[0,\infty)$. For $t\geq 0$ we define
\[
 \EE_{\{t\}}(f):=\EE_0(\EE_{0,t}(\, \cdot\, , f))\quad \mbox{for all }f\in C_b(S).
\]
For $n\in \N$, $0\leq t_1<\ldots<t_{n+1}\leq T$ and $f\in C_b(S^{n+1})$, we define recursively
\[
 \EE_{\{t_1,\ldots, t_{n+1}\}}(f):=\EE_{\{t_1,\ldots, t_n\}}(g),
\]
where $g(x_1,\ldots, x_n):=\EE_{t_n,t_{n+1}}(x_n,f(x_1,\ldots, x_n,\, \cdot \, ))$ for all $x_1,\ldots, x_n\in S$.
Note that $g\in C_b(S^n)$ by Proposition \ref{parameterdep}. Then $\EE_J\colon C_b(S^J)\to \R$ is a pre-expectation which is continuous from above for all $J\in \HH$. By the Chapman-Kolmogorov equations, we obtain that the family $(\EE_J)_{J\in \HH}$ is consistent. Therefore, by Theorem \ref{extension}, there exists a nonlinear epectation $\EE$ on the path space $(\Om,\FF):=\big(S^{[0,\infty)},\BB^{[0,\infty)}\big)$, such that $(\Om,\FF,\EE)$ is a convex expectation space and the canonical process $X_t(\omega)=\omega_t$, $t\in[0,\infty)$, satisfies (i) and (ii).
\end{proof}

\begin{example}\label{controlapp}
Let $U$ be a Polish space. Let $b\colon [0,T]\times \R^n\times U\to \R^n$ and $\si\colon [0,T]\times \R^n\times U\to \R^{n\times m}$ be uniformly continuous and assume that there is a constant $L>0$ such that for $\varphi\in \{b,\si\}$ one has
 \begin{itemize}
  \item $|\varphi(t,x_1,u)-\varphi(t,x_2,u)|\leq L|x_1-x_2|$ \;for all $t\in [0,T]$, $x_1,x_2\in \R^n$ and $u\in U$,
  \item $|\varphi(t,0,u)|\leq L$ \;for all $t\in [0,T]$ and $u\in U$.
 \end{itemize}
Following \cite[Chapter 4, Section 3]{MR1696772}, we denote by $\UU^\om([s,t])$ the set of all $5$-tuples $u^\om=(\Om,\FF,\P,W,u)$ satisfying the following:
\begin{enumerate}
 \item[(i)] $(\Om,\FF,\P)$ is a complete probability space.
 \item[(ii)] $(W_r)_{r\in [s,t]}$ is an $m$-dimensional standard Brownian motion defined on $(\Om,\FF,\P)$ over $[s,t]$ with $W_s=0$ $\P$-almost surely. Moreover, let $\FF_r^{s,t}=\si(W_\tau\colon s\leq \tau \leq r)$ augmented by all the $\P$-null sets in $\FF$ for all $s\leq r\leq t$.
 \item[(iii)] $u\colon [s,t]\times \Om\to U$ is $(\FF_r^{s,t})_{s\leq r\leq t}$-progressively measurable.
\end{enumerate}
For all $y\in \R^n$ and $u^\om=(\Om,\FF,\P,W,u)\in \UU^\om([s,t])$ let $(x_r(s,y,u^\om))_{r\in [s,t]}$ be the solution of the SDE
\[
 dx_r=b(t,x_r,u_r)dr+\si(t,x_r,u_r)dW_r,\; r\in (s,t], \quad x_s=y.
\]
We denote by $C_b^\th(\R^n)$ the space of all H\"older continuous functions with H\"older exponent $\th\in (0,1)$ and the corresponding H\"older norm by $\|\cdot\|_\th$. For $f\in C_b^\th(\R^n)$, $y\in \R^n$ and $u^\om=(\Om,\FF,\P,W,u)\in \UU^\om([s,t])$ we define
\[
\mu^{u^\om}_{s,t}(y,f):=\E_\P(f(x_t(s,y,u^\om))).
\]
Then, by H\"older's inequality with $p=\frac{1}{\th}$ and Theorem 6.16 in \cite[Chapter 1, p. 49]{MR1696772}, for all $f\in C_b^\th(\R^n)$ we have that
\begin{align*}
  |\mu_{s,t}^{u^\om}(y,f)-\mu_{s,t}^{u^\om}(z,f)|&\leq \E_\P\big(\big|f(x_t(s,y,u^\om))-f(x_t(s,z,u^\om))\big|\big)\\
  &\leq \|f\|_\th\E_\P(|x_t(s,y,u^\om)-x_t(s,z,u^\om)|^\th)\\
  &\leq \|f\|_\th\E_\P(|x_t(s,y,u^\om)-x_t(s,z,u^\om)|)^\th\\
  & \leq \|f\|_\th L_{s,t} |y-z|^\th
\end{align*}
for all $y,z\in \R^n$, where $L_{s,t}$ is independent of $y,z\in \R^n, f\in C_b^\th(S)$ and $u^\om\in \UU^\om([s,t])$. Hence, $\mu^{u^\om}_{s,t}$ defines a linear pre-kernel from $C_b^\th(\R^n)$ to $C_b^\th(\R^n)$. By Chebychev's inequality and Theorem 6.16 in \cite[Chapter 1, p. 49]{MR1696772} we get that
\[
 \P(|x_t(s,y,u^\om)|>M)\leq \frac{\E_\P(|x_t(s,y,u^\om)|^2)}{M^2}\leq \frac{C_T(1+|y|^2)|t-s|}{M^2}
\]
with a constant $C_T>0$ independent of $u^\om\in \UU^\om([s,t])$. Hence, the family $\{\mu^{u^\om}_{s,t}(y,\, \cdot\, )\colon u^\om\in \UU^\om([s,t])\}$ is tight. For all $f\in C_b^\th(\R^n)$ we define
\[
 \EE_{s,t}(y,f):=\sup_{u^\om\in \UU^\om([s,t])}\mu^{u^\om}_{s,t}(y,f).
\]
Therefore, $\EE_{s,t}$ defines a pre-kernel from $C_b^\th(\R^n)$ to $C_b^\th(\R^n)$, which is continuous from above and the dynamic programming principle \cite[Chapter 4, Theorem 3.3, p. 180]{MR1696772} implies that the family $(\EE_{s,t})_{0\leq s<t\leq T}$ satisfies the Chapman-Kolmogorov equations. Hence, by Theorem \ref{chapman} there exists a nonlinear expectation space $(\Om,\FF,\EE)$ and a stochastic process $(X_t)_{t\geq 0}$ of random variables $\Om\to \R^n$ which satisfies (i) and (ii) in Theorem \ref{chapman}. If $U\subset \R^{n\times n}$ is a compact nonempty subset of positive definite matrices, $b\equiv 0$ and $\si(t,x,u)=u$ the expectation $\EE$ coincides with the G-expectation introduced by Peng \cite{PengG},\cite{MR2474349}.
\end{example}


\begin{appendix}\section{}\label{appA}
In this appendix, we provide the proof of Lemma \ref{2.5} and state three other technical lemmas.
\begin{proof}[Proof of Lemma \ref{2.5}]
For $\alpha\in\R$, let $\PP_\alpha$ be the set of all linear functions $\mu\colon M\to \R$ with $\EE^*(\mu)\le \alpha$. For each $\mu\in\PP_\al$ and every $\la >0$ one has
\begin{align*}
 1-\la^{-1} \EE^*(\mu)&= -\la^{-1}(\EE(-\la)+\EE^*(\mu))\\
 &\leq -\la^{-1} \mu(-\la)=\mu 1=\la^{-1} \mu(\la)\\
 &\leq \la^{-1} (\EE(\la)+ \EE^*(\mu))= 1+\la^{-1} \EE^*(\mu),
\end{align*}
and therefore, letting $\la \to \infty$, we obtain $\mu 1=1$. Further, for all $X,Y\in M$ with $X\leq Y$ one has
\[
 \mu(X-Y)\leq \la^{-1}(\EE(\la (X-Y))+\EE^*(\mu))\leq \la^{-1}\EE^*(\mu)\to 0,\quad \la\to \infty.
\]
Hence, $\mu\colon M\to \R$ is a linear pre-expectation on $M$ and therefore continuous. Thus,
\[
 \PP_\al=\bigcap_{X\in M}\{\mu\in M'\colon \mu X\leq \EE(X)+\al\}
\]
is convex and a closed subset of the compact unit ball 
and therefore compact.

We next show \eqref{rep:M}. The inequality ``$\ge$'' follows by definition of $\EE^\ast$.
Fix $X\in M$ and
 let $\EE_0(\al):=\EE(\al X)$ for all $\al\in \R$. Then $\EE_0\colon \R\to \R$ is convex. Hence, there exists
 $m \in \R$ such that
 \[
  \EE(\al X)=\EE_0(\al)\geq \EE_0(1)+ m (\al -1)=(\EE(X)-m)+ m \al
 \]
 for all $\al \in \R$. By the theorem of Hahn-Banach there exists a linear functional $\mu\colon M\to\mathbb{R}$ such that
 \[
  \EE(Y)\geq (\EE(X)-m) +\mu Y
 \]
 for all $Y\in M$ and $\mu(\al X)= m \al$ for all $\al \in \R$. Hence,
 \[
 \mu Y-\EE(Y)\leq m -\EE(X)=:c
 \]
for all $Y\in M$ and $\mu X-\EE(X)=c$ so that $c=\EE^*(\mu)$. Thus, $\mu\in M'$ by the first part of the proof
and $\EE(X)=\mu X- \EE^*(\mu)$. For each $\mu\in M'$ with $\EE^*(\mu) > \alpha:= \|X\|_\infty-\EE(X)$ we have
\[ \mu X-\EE^*(\mu)\leq  \|X\|_\infty-\EE^*(\mu) < \|X\|_\infty -\alpha = \EE(X).\]
Therefore, the maximum in \eqref{rep:M} is attained on the set $\PP_\alpha$.

Finally, if $\EE$ is sublinear
let $\mu\in M'$ with $\EE^*(\mu)<\infty$. For each $X\in M$ and every $\la>0$ one has
\[
 \la (\mu X-\EE(X))=\mu (\la X)-\EE(\la X)\leq \EE^*(\mu)<\infty
\]
so that $\mu X-\EE(X)\leq 0$. Since $\EE(0)=0$, we obtain $\EE^*(\mu)=\sup_{X\in M}\left(\mu X-\EE(X)\right)=0$.
\end{proof}

\begin{lemma}\label{transfo}
 Let $M\subset \LL^\infty(\Om,\FF)$ be a linear subspace with $1\in M$, $\EE\colon M\to \R$ be a convex pre-expectation, $\Om_0\neq \es$ and $T\colon \Om\to \Om_0$ be an arbitrary mapping. Further, let $M_0\subset \LL^\infty(\Om_0,2^{\Om_0})$ be a linear subspace with $1\in M_0$ and $M_0\circ T:=\{Y\circ T\colon Y\in M_0\}\subset M$. Then,
 \[
   \EE\circ T^{-1}\colon M_0\to \R,\quad Y\mapsto \EE(Y\circ T)
 \]
defines a convex pre-expectation on $M_0$. If $\EE$ is sublinear, then $\EE\circ T^{-1}$ is sublinear and we have that
\[
 \{\nu\in M_0'\colon (\EE\circ T^{-1})^*(\nu)=0\}=\{\mu\circ T^{-1}\colon \mu\in M', \; \EE^*(\mu)=0\}.
\]
\end{lemma}

\begin{proof}
 It is easily verified that $\EE\circ T^{-1}$ defines a convex pre-expectation on $M_0$. Let $\mu\in M'$ with $\EE^*(\mu)<\infty$. Then, we have that
  \[
   (\mu\circ T^{-1})(Y)-(\EE\circ T^{-1})(Y)=\mu (Y\circ T)-\EE(Y\circ T)\leq \EE^*(\mu).
  \]
  for all $Y\in M_0$. Hence, $(\EE\circ T^{-1})(\mu\circ T^{-1})\leq \EE^*(\mu)=0$. As the mapping $M'\to M_0',\; \mu\mapsto \mu\circ T^{-1}$ is continuous, we have that
  \[
   \{\mu\circ T^{-1}\colon \mu\in M', \; \EE^*(\mu)=0\}
  \]
  is compact. By the separation theorem of Hahn-Banach, it follows
  \[
   \{\nu\in M_0'\colon (\EE\circ T^{-1})^*(\nu)=0\}=\{\mu\circ T^{-1}\colon \mu\in M', \; \EE^*(\mu)=0\}.
  \]
\end{proof}

\end{appendix}

\bibliographystyle{amsplain}

\begin{thebibliography}{99}

\bibitem{MR1850791}
P.~Artzner, F.~Delbaen, J.-M. Eber, and D.~Heath.
\newblock Coherent measures of risk.
\newblock {\em Math. Finance}, 9(3):203--228, 1999.

\bibitem{dbartl}
D.~Bartl.
\newblock Pointwise dual representation of dynamic convex expectations.
\newblock {Preprint}, 2016.

\bibitem{MR0149533}
D.~Bierlein.
\newblock \"{U}ber die {F}ortsetzung von {W}ahrscheinlichkeitsfeldern.
\newblock {\em Z. Wahrscheinlichkeitstheorie und Verw. Gebiete}, 1:28--46,
  1962/1963.

\bibitem{MR2267655}
V.~I. Bogachev.
\newblock {\em Measure theory. {V}ol. {I}}.
\newblock Springer-Verlag, Berlin, 2007.

\bibitem{MR3153589}
S.~Cerreia-Vioglio, F.~Maccheroni, M.~Marinacci, and A.~Rustichini.
\newblock Niveloids and their extensions: risk measures on small domains.
\newblock {\em J. Math. Anal. Appl.}, 413(1):343--360, 2014.

\bibitem{MR2199055}
P.~Cheridito, F.~Delbaen, and M.~Kupper.
\newblock Dynamic monetary risk measures for bounded discrete-time processes.
\newblock {\em Electron. J. Probab.}, 11:no. 3, 57--106, 2006.

\bibitem{CKT}
P.~Cheridito, M.~Kupper, and L.~Tangpi.
\newblock Representation of increasing convex functionals with countably
  additive measures.
\newblock {\em Preprint}, 2015.

\bibitem{CKT2}
P.~Cheridito, M.~Kupper, and L.~Tangpi.
\newblock Duality formulas for robust pricing and hedging in discrete time.
\newblock Forthcoming in {\em  SIAM Journal of Financial Mathematics}, 2017.

\bibitem{MR2319056}
P.~Cheridito, H.~M. Soner, N.~Touzi, and N.~Victoir.
\newblock Second-order backward stochastic differential equations and fully
  nonlinear parabolic {PDE}s.
\newblock {\em Comm. Pure Appl. Math.}, 60(7):1081--1110, 2007.

\bibitem{MR0112844}
G.~Choquet.
\newblock Forme abstraite du th\'eor\`eme de capacitabilit\'e.
\newblock {\em Ann. Inst. Fourier. Grenoble}, 9:83--89, 1959.

\bibitem{CHMP}
F.~Coquet, Y.~Hu, J.~M\'emin, and S.~Peng.
\newblock Filtration-consistent nonlinear expectations and related
  g-expectations.
\newblock {\em Probability Theory and Related Fields}, pages 123:1--27, 2002.

\bibitem{MR2011534}
F.~Delbaen.
\newblock {\em Coherent risk measures}.
\newblock Cattedra Galileiana. [Galileo Chair]. Scuola Normale Superiore,
  Classe di Scienze, Pisa, 2000.

\bibitem{MR1929369}
F.~Delbaen.
\newblock Coherent risk measures on general probability spaces.
\newblock In {\em Advances in finance and stochastics}, pages 1--37. Springer,
  Berlin, 2002.

\bibitem{MR2276899}
F.~Delbaen.
\newblock The structure of m-stable sets and in particular of the set of risk
  neutral measures.
\newblock In {\em In memoriam {P}aul-{A}ndr\'e {M}eyer: {S}\'eminaire de
  {P}robabilit\'es {XXXIX}}, volume 1874 of {\em Lecture Notes in Math.}, pages
  215--258. Springer, Berlin, 2006.

\bibitem{MR2670421}
F.~Delbaen, S.~Peng, and E.~Rosazza~Gianin.
\newblock Representation of the penalty term of dynamic concave utilities.
\newblock {\em Finance Stoch.}, 14(3):449--472, 2010.

\bibitem{MR521810}
C.~Dellacherie and P.-A. Meyer.
\newblock {\em Probabilities and potential}, volume~29 of {\em North-Holland
  Mathematics Studies}.
\newblock North-Holland Publishing Co., Amsterdam-New York; North-Holland
  Publishing Co., Amsterdam-New York, 1978.

\bibitem{MR2754968}
L.~Denis, M.~Hu, and S.~Peng.
\newblock Function spaces and capacity related to a sublinear expectation:
  application to {$G$}-{B}rownian motion paths.
\newblock {\em Potential Anal.}, 34(2):139--161, 2011.

\bibitem{MR2868935}
Y.~Dolinsky, M.~Nutz, and H.~M. Soner.
\newblock Weak approximation of {$G$}-expectations.
\newblock {\em Stochastic Process. Appl.}, 122(2):664--675, 2012.

\bibitem{MR1009162}
N.~Dunford and J.~T. Schwartz.
\newblock {\em Linear operators. {P}art {I}}.
\newblock Wiley Classics Library. John Wiley \& Sons, Inc., New York, 1988.
\newblock General theory, With the assistance of William G. Bade and Robert G.
  Bartle, Reprint of the 1958 original, A Wiley-Interscience Publication.

\bibitem{MR0055678}
K.~Fan.
\newblock Minimax theorems.
\newblock {\em Proc. Nat. Acad. Sci. U. S. A.}, 39:42--47, 1953.

\bibitem{MR2179357}
W.~H. Fleming and H.~M. Soner.
\newblock {\em Controlled {M}arkov processes and viscosity solutions},
  volume~25 of {\em Stochastic Modelling and Applied Probability}.
\newblock Springer, New York, second edition, 2006.

\bibitem{follmer2014spatial}
H.~F{\"o}llmer and C.~Kl{\"u}ppelberg.
\newblock Spatial risk measures: Local specification and boundary risk.
\newblock In {\em Stochastic Analysis and Applications 2014}, pages 307--326.
  Springer, 2014.

\bibitem{MR2323189}
H.~F{\"o}llmer and I.~Penner.
\newblock Convex risk measures and the dynamics of their penalty functions.
\newblock {\em Statist. Decisions}, 24(1):61--96, 2006.

\bibitem{MR2779313}
H.~F{\"o}llmer and A.~Schied.
\newblock {\em Stochastic finance}.
\newblock Walter de Gruyter \& Co., Berlin, extended edition, 2011.
\newblock An introduction in discrete time.

\bibitem{MR1725607}
D.~J. Hartfiel.
\newblock {\em Markov set-chains}, volume 1695 of {\em Lecture Notes in
  Mathematics}.
\newblock Springer-Verlag, Berlin, 1998.

\bibitem{MR2143645}
S.~Peng.
\newblock Nonlinear expectations and nonlinear {M}arkov chains.
\newblock {\em Chinese Ann. Math. Ser. B}, 26(2):159--184, 2005.

\bibitem{PengG}
S.~Peng.
\newblock G-expectation, {G}-{B}rownian motion and related stochastic calculus
  of {I}t\^o type.
\newblock {\em Stochastic Analysis and Applications}, Volume 2 of Abel
  Symposium:541--567, 2007.

\bibitem{MR2474349}
S.~Peng.
\newblock Multi-dimensional {$G$}-{B}rownian motion and related stochastic
  calculus under {$G$}-expectation.
\newblock {\em Stochastic Process. Appl.}, 118(12):2223--2253, 2008.

\bibitem{MR2746175}
H.~M. Soner, N.~Touzi, and J.~Zhang.
\newblock Martingale representation theorem for the {$G$}-expectation.
\newblock {\em Stochastic Process. Appl.}, 121(2):265--287, 2011.

\bibitem{MR2842089}
H.~M. Soner, N.~Touzi, and J.~Zhang.
\newblock Quasi-sure stochastic analysis through aggregation.
\newblock {\em Electron. J. Probab.}, 16:no. 67, 1844--1879, 2011.

\bibitem{MR0224522}
B.~Z. Vulikh.
\newblock {\em Introduction to the theory of partially ordered spaces}.
\newblock Translated from the Russian by Leo F. Boron, with the editorial
  collaboration of Adriaan C. Zaanen and Kiyoshi Is\'eki. Wolters-Noordhoff
  Scientific Publications, Ltd., Groningen, 1967.

\bibitem{MR1696772}
J.~Yong and X.~Y. Zhou.
\newblock {\em Stochastic controls}, volume~43 of {\em Applications of
  Mathematics (New York)}.
\newblock Springer-Verlag, New York, 1999.
\newblock Hamiltonian systems and HJB equations.

\end{thebibliography}

\end{document}